\documentclass[12pt, letterpaper]{article}

\usepackage{multirow} 
\usepackage{enumitem} 
\usepackage{amsfonts} 
\usepackage[centertags]{amsmath}
\usepackage{amssymb}
\usepackage{amsthm}
\usepackage{algorithm}
\usepackage{algorithmic}
\newtheorem{thm}{Theorem}[section]
\newtheorem{lem}[thm]{Lemma}

\newtheorem{cor}[thm]{Corollary}

\newtheorem{definition}[thm]{Definition}
\usepackage[colorinlistoftodos]{todonotes} 
\usepackage[numbers, sort&compress]{natbib}
\usepackage{cancel}
\usepackage{xcolor}

\title{Piecewise M-Stationarity and Convergence of a Bounding Algorithm for  Mathematical Programs with Complementarity Constraints}

\author{Kexin Wang  
\thanks{College of Control Science and Engineering, Zhejiang
  University, Hangzhou, Zhejiang 310027, China, kxwang@zju.edu.cn} 
and Lorenz T. Biegler
\thanks{Department of Chemical Engineering, Carnegie Mellon
  University, 5000 Forbes Ave, Pittsburgh, PA 15213, USA,
  lb01@andrew.cmu.edu (corresponding author)}}

\begin{document}
\maketitle

\begin{abstract}
This study explores B-stationarity condition for mathematical programs with
complementarity constraints (MPCCs)
and convergence of a Bounding Algorithm we proposed in \cite{bounding}. 
Special attention is given to the cases
with biactive complementarity constraints.
First, we introduce the concept of piecewise M-stationarity and prove its
equivalence to B-stationarity under MPCC-GCQ.
Then, this concept is employed to propose a simplified B-stationarity verification.
Finally, we investigate improved convergence properties of the Bounding Algorithm in \cite{bounding}.
Limit points of this method are shown to be
C-stationary, based on the stationarity of an approaching sequence; moreover,
they are piecewise M-stationary (and therefore B-stationary), if an additional
second-order condition is satisfied by the sequence. 
\end{abstract}

\noindent
{\bf Keywords:} Complementarity constraints; MPCCs; B-stationarity; Stationarity  conditions; Constraint qualifications; Smoothed NCP functions

\section{Introduction}

We consider mathematical programs with complementarity constraints
(MPCCs) of the form 
\begin{equation} \label{mpcc}
\begin{aligned}
\min \quad &f(z) \\
{\rm s.t.} \quad &g(z) \leq 0 ,\\
~ &h(z) = 0 ,\\ 
~ &0 \leq G(z) \perp H(z) \geq 0, 
\end{aligned}
\end{equation}
where $(f,g,h,G,H): \mathbb{R}^n \to \mathbb{R}^{1+n_g+n_h+m+m}$ are twice continuously differentiable functions.
At a feasible point $\bar z$ of the MPCC, define the following index sets: 
\begin{equation} \label{indsets}
\begin{aligned}
  I_g(\bar z) &= \{i \,|\, g_i(\bar z) = 0\},\\
  I_h(\bar z) &= \{i \,|\, i = 1, \dots, n_h\},\\
  \alpha(\bar z) &= \{i \,|\, G_i(\bar z) = 0, H_i(\bar z) > 0\},\\
  \gamma(\bar z) &= \{i \,|\, G_i(\bar z) > 0, H_i(\bar z) = 0\},\\
  \beta(\bar z) &= \{i \,|\, G_i(\bar z) = 0,  H_i(\bar z) = 0\}.
\end{aligned}
\end{equation}
A feasible point $\bar z$ is weakly stationary, if there exist
multipliers $\bar \lambda =
(\bar\lambda^g,\bar\lambda^h,\bar\lambda^G,\bar\lambda^H)$ with
$\bar\lambda^g \geq 0$, 
such that
\begin{equation} \label{wstat}
\begin{aligned}
0 = &\nabla f(\bar z) 
 + \sum_{i \in I_g(\bar z)} \bar\lambda^g_i \nabla g_i(\bar z)
 + \sum_{i \in I_h(\bar z)} \bar\lambda_i^h \nabla h_i(\bar z) \\
 &- \sum_{i \in \alpha(\bar z) \cup \beta(\bar z)} \bar\lambda^G_i
 \nabla G_i(\bar z) 
 - \sum_{i \in \gamma(\bar z) \cup \beta(\bar z)} \bar\lambda^H_i
 \nabla H_i(\bar z).
\end{aligned}
\end{equation}
Further, a weakly stationary point $\bar z$ is also
\begin{itemize}
  \item
    S-stationary (strongly stationary), if
    $\bar\lambda_i^G,\bar\lambda_i^H \geq 0$ for all 
    $i \in \beta(\bar z)$;
  \item
    M-stationary, if either
    $\bar\lambda_i^G,\bar\lambda_i^H \geq 0$ or $\bar\lambda_i^G
    \bar\lambda_i^H = 0$ for all $i \in \beta(\bar z)$;
  \item
    C-stationary, if $\bar\lambda_i^G\bar\lambda_i^H \geq 0$ for all
    $i \in \beta(\bar z)$;
  \item
    A-Stationary, if either
    $\bar\lambda_i^G \geq 0$ or $\bar\lambda_i^H \geq 0$
    for all $i \in \beta(\bar z)$.
\end{itemize}

\subsection{First-Order Optimality Conditions and Constraint Qualifications}

A local minimizer $\bar z$ of MPCC (\ref{mpcc}) is a
B-stationary point at which the following condition holds
\begin{equation} \label{bstat}
  \nabla f(\bar z)^Td \geq 0, \quad\forall d \in \mathcal T(\bar z),
\end{equation}
where $\mathcal T(\bar z)$ is the tangent cone of the MPCC feasible set at $\bar z$.
This condition is the same as (see also \cite[Theorem 6.12]{rockafellar-wets})
\begin{equation} \label{bstat-d}
  \nabla f(\bar z) \in \mathcal T(\bar z)^{*},
\end{equation}
where $\mathcal T(\bar z)^{*}$ is the dual cone of $\mathcal T(\bar
z)$ and is defined by
\begin{equation*}
\mathcal T(\bar z)^{*}
   = \{ w \,|\, w^T d \geq 0, \;
      \forall d \in \mathcal T(\bar z) \}.
\end{equation*}
Without an analytic expression of the tangent cone, 
verifying these conditions directly is generally
nontrivial.
In practice, it is desirable to employ linearized
cones to reconstruct the first-order condition (\ref{bstat}) or
(\ref{bstat-d}). Constraint 
qualifications (CQs) play an important role in this task.

Standard linearization of $\mathcal T(\bar z)$ can be carried out (see
\cite[Eqs. (10)-(11)]{mpec-acq}), by 
replacing the complementarity constraints $0 \leq G(z) \perp H(z) \geq
0$ with
\begin{equation*}
G(z) \geq 0,\quad H(z) \geq 0,\quad G(z)^TH(z) = 0.
\end{equation*}
Linearizing these constraints gives that
\begin{equation*}
\begin{aligned}
    G_i(\bar z) + \nabla G_i(\bar z)^Td \geq 0, \quad &i = 1, \dots, m,\\
    H_i(\bar z) + \nabla H_i(\bar z)^Td \geq 0, \quad &i = 1, \dots, m,\\
    G_i(\bar z)H_i(\bar z) + H_i(\bar z)\nabla G_i(\bar z)^Td + G_i(\bar z)\nabla H_i(\bar z)^Td = 0,
                            \quad &i = 1, \dots, m,
\end{aligned}
\end{equation*}
which, together with the index
sets in (\ref{indsets}), gives the following linearized tangent cone:
\begin{equation*}
\begin{aligned}
    \mathcal T^{\rm lin}(\bar z) = \{ d \,|
      & \nabla g_i(\bar z)^Td \leq 0, &&\forall i \in I_g(\bar z),\\
    ~ & \nabla h_i(\bar z)^Td = 0, &&\forall i \in I_h(\bar z),\\
    ~ & \nabla G_i(\bar z)^Td = 0, &&\forall i \in \alpha(\bar z),\\
    ~ & \nabla H_i(\bar z)^Td = 0, &&\forall i \in \gamma(\bar z),\\
    ~ & \nabla G_i(\bar z)^Td \geq 0, \; \nabla H_i(\bar z)^Td \geq 0,
      &&\forall i \in \beta(\bar z) \}.
\end{aligned}
\end{equation*}
Its dual cone is given by
\begin{equation*}
\begin{aligned}
  \mathcal T^{\rm lin}(\bar z)^{*}
  &= \{ w \,|\, w^T d \geq 0,
     \; \forall d \in \mathcal T^{\rm lin}(\bar z) \}\\
  &\begin{aligned}\;= \{ w \,|\,
  0&= w
    + \nabla g_I(\bar z)\eta_I^g 
    + \nabla h(\bar z) \eta^h
    - \eta_{\alpha}^G \nabla G_{\alpha}(\bar z)
    - \eta_{\gamma}^H \nabla H_{\gamma}(\bar z) \\
  ~&- \eta_{\beta}^G \nabla G_{\beta}(\bar z)
    - \eta_{\beta}^H \nabla H_{\beta}(\bar z),
    \; \eta_I^g, \eta_{\beta}^G, \eta_{\beta}^H \geq 0\} ,
  \end{aligned} 
\end{aligned}
\end{equation*}
where
$g_I$ denotes the constraints $\{g_i \,|\, \forall i \in
I_g(\bar z)\}$, and similarly, $G_{\alpha}, H_{\gamma}, G_{\beta}$, and
$H_{\beta}$ denote the constraints related to the index sets
$\alpha(\bar z), \gamma(\bar z)$, and $\beta(\bar z)$.
The standard Abadie and Guignard
constraint qualifications, NLP-ACQ and NLP-GCQ, assume
$\mathcal T(\bar z) = \mathcal T^{\rm lin}(\bar
z)$ and $\mathcal T(\bar z)^{*} = \mathcal T^{\rm lin}(\bar
z)^{*}$, respectively, so that
the conditions (\ref{bstat}) and (\ref{bstat-d}) can be rebuilt based on the
linearized cone.
This converts the first-order stationarity of MPCC (\ref{mpcc}) into that of
the relaxed NLP 
\begin{equation}\label{rnlp}
\begin{aligned}
{\rm RNLP}: \quad \min \quad &f(z) \\
{\rm s.t.} \quad &g(z) \leq 0, \\
  ~ &h(z) = 0,\\
  ~ &G_i(z) = 0, &&i \in \alpha(\bar z),\\
  ~ &H_i(z) = 0, &&i \in \gamma(\bar z),\\
  ~ &G_i(z) \geq 0, \; H_i(z) \geq 0, &&i \in \beta(\bar z),
\end{aligned}
\end{equation}
and thus justifies using the KKT
conditions for RNLP, i.e., the S-stationarity condition, 
as a necessary first-order condition for the MPCC (see also
\cite[Theorem 4.1]{gcq-in-mpcc}).
However, we should note that NLP-ACQ cannot hold if $\beta(\bar z) \neq
\emptyset$, because $\mathcal T(\bar z)$ is a nonconvex cone in this case
while $\mathcal T^{\rm lin}(\bar z)$ is always a convex polyhedral cone.

Since NLP-CQs are usually too restrictive and not expected to hold in
the presence of complementarity constraints,
constraint qualifications tailored to MPCCs
have been proposed.  
MPCC-ACQ \cite{mpec-acq} and MPCC-GCQ \cite{mstat-proof} extend the standard Abadie and Guignard constraint qualifications for NLPs to CQs that are tailored to the geometry of MPCCs.
Here, instead of relating the tangent cone $\mathcal
T(\bar z)$ with $\mathcal T^{\rm lin}(\bar z)$,  $\mathcal T(\bar z)$ 
is related with the MPCC-linearized tangent cone $\mathcal T_{\rm MPCC}^{\rm
  lin}(\bar z)$, which is given by:
\begin{equation*}\label{tlin-mpcc}
\begin{aligned}
    \mathcal T_{\rm MPCC}^{\rm lin}(\bar z) = \{ d \,|
      & \nabla g_i(\bar z)^Td \leq 0, &&\forall i \in I_g(\bar z),\\
    ~ & \nabla h_i(\bar z)^Td = 0, &&\forall i \in I_h(\bar z),\\
    ~ & \nabla G_i(\bar z)^Td = 0, &&\forall i \in \alpha(\bar z),\\
    ~ & \nabla H_i(\bar z)^Td = 0, &&\forall i \in \gamma(\bar z),\\
    ~ & \nabla G_i(\bar z)^Td \geq 0, &&\forall i \in \beta(\bar z),\\
    ~ & \nabla H_i(\bar z)^Td \geq 0, &&\forall i \in \beta(\bar z),\\
    ~ & (\nabla G_i(\bar z)^Td) \cdot (\nabla H_i(\bar z)^Td) = 0,
                                   &&\forall i \in \beta(\bar z) \}.
\end{aligned}
\end{equation*}
MPCC-ACQ assumes $\mathcal T(\bar z)
= \mathcal T_{\rm MPCC}^{\rm lin}(\bar z)$, 
then the condition (\ref{bstat}) can be expressed as:
\begin{equation*}
\nabla f(\bar z)^T d \geq 0, \quad\forall d \in \mathcal T_{\rm
  MPCC}^{\rm lin}(\bar z).
\end{equation*}
MPCC-GCQ assumes
$\mathcal T(\bar 
z)^{*} = \mathcal T_{\rm MPCC}^{\rm lin}(\bar z)^{*}$,
then the condition (\ref{bstat-d}) can be expressed as:
\begin{equation*}
\nabla f(\bar z) \in \mathcal T_{\rm MPCC}^{\rm lin}(\bar z)^{*}.
\end{equation*}
MPCC-GCQ is implied by MPCC-ACQ, but the converse is in general not
true. Their relations are analogous to the relations between
NLP-GCQ and NLP-ACQ. Examples showing that NLP-GCQ and 
MPCC-GCQ have a better chance to be satisfied, even if NLP-ACQ and
MPCC-ACQ do not hold, can be found in \cite[Example 1.3]{lecturenote} and
\cite[Example 2.1]{mstat-proof}.
Note that despite the fact that a tangent cone is
not necessarily equal to the closure of its convex hull, their dual
cones are the same. 
This offers the
opportunity for NLP-GCQ and MPCC-GCQ to hold more generally.
It has been established that under
MPCC-GCQ, M-stationarity is a necessary
first-order condition \cite[Theorem 3.1]{mstat-proof}.
In addition,
Fritz John type M-stationarity has been derived
at a local minimizer of an MPCC without requiring a constraint qualification
\cite[Theorem 3.1]{mpecfj}.

\subsection{Degeneracy}

To seek a solution of
MPCC (\ref{mpcc}),
various NLP-based schemes have been proposed.
The original intention is to avoid dealing with the complementarity
structure explicitly. 
In general, these schemes are designed to 
solve a sequence of regularized NLPs, yielding a sequence of stationary
points $\{z^k\}$ that is hoped to approximate a solution of the MPCC. 
An important ingredient is to characterize conditions under which,
as the regularization parameter vanishes or stabilizes,
a limit point of a sequence $\{z^k\}$ is a stationary point of the
MPCC in some sense.  
For some representative work, see 
\cite{reg,fukushima-pang,lin-fukushima,IPM-for-mpcc,reform-mstat1,reform-mstat2,reform,sqp,augLag2}. 

A difficulty arises in establishing stationarity of a limit point
that is degenerate (on the lower level), namely,
a sequence $\{z^k\} \to 
\bar z$ at which $\beta(\bar z) \neq \emptyset$. 
Fukushima and Pang \cite{fukushima-pang} studied the behavior of a
sequence $\{z^k\}$ that is composed of KKT points of NLPs formulated by
smoothing the MPCC with perturbed
Fischer-Burmeister functions. The condition of
\emph{asymptotic weak nondegeneracy} was proposed, meaning that
for every $i \in \beta(\bar z)$, $G_i(z^k)$ and $H_i(z^k)$ approach
zero at the same order of magnitude.
Under this condition and
second-order necessary conditions at every $z^k$,
together with MPCC linear independence constraint qualification
(MPCC-LICQ) at $\bar z$,
it has been proved that $\bar z$ is a
B-stationary point of the MPCC \cite[Theorem
3.1]{fukushima-pang}. However, the condition of
asymptotic weak nondegeneracy is hard to enforce in practice.
Replacing this condition with upper level strict
complementarity (ULSC),
namely, $\bar \lambda_i^G \bar \lambda_i^H \neq 0$ for all $i \in
\beta(\bar z)$,
Scholtes recovered B-stationarity of a limit point of a
regularization scheme \cite[Corollary 3.4]{reg}.
Kadrani et al.
developed a regularization method whose limit points were shown to be
M-stationary under MPCC-LICQ,
and S-stationary under additional assumption of asymptotic weak
nondegeneracy \cite{reform-mstat1}.
The result on M-stationarity was later proved valid under 
the weaker MPCC constant positive linear dependence
(MPCC-CPLD) assumption \cite{compare}. 
Results under weaker assumptions also include, for example,
C-stationarity convergence of the method by Steffensen and Ulbrich
under MPCC constant rank constraint qualification (MPCC-CRCQ)
\cite{reform} and under MPCC-CPLD \cite{cpld-mpvc}, and M-stationarity 
convergence of the method by Kanzow and Schwartz under MPCC-CPLD
\cite{reform-mstat2}.
Theoretical and numerical comparison of some of these methods can be
found in \cite{compare}.

Besides diverse methods for reformulating complementarity constraints, many 
popular algorithmic frameworks in nonlinear programming have been
exploited to deal with 
complementarity as well as the potential degeneracy.
Sequential quadratic programming (SQP) methods applied to MPCCs were 
investigated in \cite{sqp}. By introducing slack
variables into the reformulation of general complementarity constraints,
superlinear convergence to a S-stationary point was
established under MPCC-LICQ and regularity conditions (Theorems 5.7
and 5.14 therein). 
An alternative SQP method that retained superlinear convergence
while relaxing some of the assumptions was analyzed in \cite{sqp2}, where
an adaptive elastic mode was invoked to enforce either feasibility of the QP
subproblems or complementarity at the iterates (Theorems 4.5 and 4.6
therein).
Interior-penalty methods for MPCCs
were studied in \cite{IPM-for-mpcc}; global convergence to a
S-stationary point was proved under MPCC-LICQ and a
condition on the 
behavior of the penalty parameters (Theorem 3.4 and Corollary 3.5 therein);
superlinear convergence to a S-stationary point was
proved under certain regularity conditions (Theorem 4.5 therein).
In particular, relations between interior-penalty and
interior-relaxation methods were established, which allow to extend
some convergence results derived for one approach to the other.
Convergence of augmented Lagrangian methods was
investigated under MPCC-LICQ \cite[Theorem 3.2]{augLag1},
where a limit point 
was proved to be S-stationary in the case of bounded multiplier
sequences, and C-stationary in the presence of unbounded multiplier sequences.
The results were improved in \cite{augLag2} for a second-order method
(Theorem 3.2 therein), where
S-stationarity was established under a weaker MPCC-relaxed constant
positive linear dependence (MPCC-RCPLD) condition, and convergence in the
presence of unbounded multipliers was proved to be
M-stationary under MPCC-LICQ.
Comparison of some existing augmented Lagrangian methods for MPCCs can
be found in \cite{augLag3}. 
Finally, a detailed survey on these MPCC topics can be found in Kim et
al. \cite{bilevel-opt}.

\subsection{Outline}

With a focus on the theoretically interesting case where $\beta(\bar
z) \neq \emptyset$, this paper investigates stationarity
conditions for MPCCs and analyzes convergence of a NLP-based
MPCC method. 
In Section \ref{sec:nec}, we propose a condition of \emph{piecewise
M-stationarity} and prove its equivalence to B-stationarity under
MPCC-GCQ.
In Section \ref{sec:verify}, we take advantage of the
piecewise M-stationarity concept to reduce the cost of
B-stationarity verification at a point that is not S-stationary.
In Section \ref{sec:ncp}, we analyze C- and piecewise M- (and
therefore B-) stationarity convergence of the
Bounding Algorithm we proposed in \cite{bounding}.

\section{Piecewise M-Stationarity} \label{sec:nec} 

\subsection{Definition}
Given a feasible point $\bar z$ of MPCC (\ref{mpcc}), denote the set
of all partitions of $\beta(\bar z)$ by $\mathcal 
P(\beta(\bar z)) = \{(\beta_1,\beta_2) \,|\,
\beta_1 \cap \beta_2 =\emptyset,
\beta_1 \cup \beta_2 = \beta(\bar z) \}$.
The branch NLP problem on every partition $(\beta_1,\beta_2) \in \mathcal
P(\beta(\bar z))$ is defined as
\begin{equation} \label{b1b2}
\begin{aligned}
{\rm NLP}_{(\beta_1,\beta_2)}: \quad
  \min \quad &f(z) \\
{\rm s.t.} \quad &g(z) \leq 0, \\
  ~ &h(z) = 0,\\
  ~ &G_i(z) = 0, &&i \in \alpha(\bar z),\\
  ~ &H_i(z) = 0, &&i \in \gamma(\bar z),\\
  ~ &G_i(z) = 0, \; H_i(z) \geq 0, &&i \in \beta_1,\\
  ~ &G_i(z) \geq 0, \; H_i(z) = 0, &&i \in \beta_2.
\end{aligned}
\end{equation}
The tangent cone of the NLP$_{(\beta_1,\beta_2)}$ feasible set at $\bar z$ is denoted by $\mathcal
T_{(\beta_1,\beta_2)}(\bar z)$, and the cone of first-order feasible
directions has the
following analytic expression:
\begin{equation*}\label{tlin}
\begin{aligned}
    \mathcal T_{(\beta_1,\beta_2)}^{\rm lin}(\bar z) = \{ d \,|\,
      &\nabla g_i(\bar z)^Td \leq 0, &&\forall i \in I_g(\bar z),\\
    ~ &\nabla h_i(\bar z)^Td = 0, &&\forall i \in I_h(\bar z),\\
    ~ &\nabla G_i(\bar z)^Td = 0, &&\forall i \in \alpha(\bar z),\\
    ~ &\nabla H_i(\bar z)^Td = 0, &&\forall i \in \gamma(\bar z),\\
    ~ &\nabla G_i(\bar z)^Td = 0, \; \nabla H_i(\bar z)^Td \geq 0,
      &&\forall i \in \beta_1,\\
    ~ &\nabla G_i(\bar z)^Td \geq 0, \; \nabla H_i(\bar z)^Td = 0,
      &&\forall i \in \beta_2 \} ,
\end{aligned}
\end{equation*}
whose dual cone is given by
\begin{equation*} \label{tlin-b1b2-dual}
\begin{aligned}
  \mathcal T_{(\beta_1,\beta_2)}^{\rm lin}(\bar z)^{*}
   &= \{ w \,|\, w^T d \geq 0, \forall d \in \mathcal T_{(\beta_1,\beta_2)}^{\rm lin}(\bar z)\} \\
  &= \{ w \,|\, 0 = w
     + \nabla g_I(\bar z) \eta_I^g + \nabla h(\bar z) \eta^h 
     - \nabla G_{\alpha}(\bar z) \eta_{\alpha}^G
     - \nabla H_{\gamma}(\bar z) \eta_{\gamma}^H \\
 &- \nabla G_{\beta_1}(\bar z) \eta_{\beta_1}^G
     - \nabla H_{\beta_1}(\bar z) \eta_{\beta_1}^H 
     - \nabla G_{\beta_2}(\bar z) \eta_{\beta_2}^G
     - \nabla H_{\beta_2}(\bar z) \eta_{\beta_2}^H, \;
 \eta_I^g, \eta_{\beta_1}^H, \eta_{\beta_2}^G \geq 0 \}.
\end{aligned}
\end{equation*}
Now we characterize piecewise
M-stationarity, as well as the standard M-stationarity for purpose of
contrast, based on the following Lagrangian function: 
\begin{equation*}
\begin{aligned}
\mathcal L(z,\lambda)
  = &f(z) + \sum_{i=1}^{n_g} \lambda_i^g g_i(z) 
  + \sum_{i=1}^{n_h} \lambda_i ^h h_i(z) 
  - \sum_{i \in \alpha(z)} \lambda_i^G G_i(z) 
  - \sum_{i \in \gamma(z)} \lambda_i^H H_i(z)  \\
~&- \sum_{i \in \beta(z)} \lambda_i^G G_i(z)
  - \sum_{i \in \beta(z)} \lambda_i^H H_i(z) .
\end{aligned}
\end{equation*}

\begin{definition} \label{def:pm}
Given a feasible point $\bar z$ of MPCC (\ref{mpcc}), we say that
\begin{enumerate}[label=(\roman*)]
\item
$\bar z$ is M-stationary if there exist
multipliers $\bar \lambda$ such that
\begin{equation} \label{mstat-st}
\begin{aligned}
  \nabla_z \mathcal L(\bar z, \bar\lambda) = 0,\\
  \bar\lambda_i^g \geq 0, \; \bar\lambda_i^g g_i(\bar z) = 0,
  &\quad\forall i \in \{1,\dots,n_g\},\\
  \text{either } \bar\lambda_i^G, \bar\lambda_i^H \geq 0 \text{ or } 
  \bar\lambda_i^G \bar\lambda_i^H =0,
  &\quad\forall i \in \beta(\bar z);
\end{aligned}
\end{equation}

\item
$\bar z$ is piecewise M-stationary if
there exist multipliers $\bar \lambda$ for each branch NLP such that
\begin{subequations} \label{pw-mstat}
\begin{align}
  \nabla_z \mathcal L(\bar z, \bar\lambda) = 0,\\
  \bar\lambda_i^g \geq 0, \; \bar\lambda_i^g g_i(\bar z) = 0,
  &\; \forall i \in \{1,\dots,n_g\},\\
  \left.
  \begin{aligned}
  \text{either } \bar\lambda_i^G, \bar\lambda_i^H \geq 0 \text{ or } 
  \bar\lambda_i^G <0, \bar\lambda_i^H =0,
  &\;\forall i \in \beta_1, \\
  \text{either } \bar\lambda_i^G, \bar\lambda_i^H \geq 0 \text{ or } 
  \bar\lambda_i^G =0, \bar\lambda_i^H <0,
  &\;\forall i \in \beta_2,
  \end{aligned}
  \right\} 
  &\; \forall (\beta_1,\beta_2) \in \mathcal P(\beta(\bar z)) .
  \label{pw-m-mult}
\end{align}
\end{subequations}
\end{enumerate}
\end{definition}

Definition \ref{def:pm} (ii) indicates that at a piecewise M-stationary
point, there may exist multiple sets of multipliers, each of which
satisfies the KKT conditions of a branch problem NLP$_{(\beta_1,\beta_2)}$
and satisfies the standard M-stationarity restriction as well.
In particular, if $\bar z$ is S-stationary, we either have $\beta(\bar z)
=\emptyset$, and therefore (\ref{pw-m-mult}) holds trivially; or we
have one set of multipliers with $\bar\lambda_i^G,\bar\lambda_i^H \geq 0$ for all $i \in
\beta(\bar z)$, which satisfies (\ref{pw-m-mult}) for all the
partitions of $\beta(\bar z)$ at once.

\subsection{Illustrative Examples} \label{eg:pmstat}

\subsubsection{Non-strongly stationary minimizer}
This example shows piecewise M-stationarity at a minimizer of an MPCC
that is not strongly stationary. Consider the problem \emph{scholtes4}  
given by
\begin{equation*} \label{scholtes4}
\begin{aligned}
  \min \quad &z_1 + z_2 - z_3 \; &{\rm multipliers}\\
  {\rm s.t.} \quad
     &-4z_1 + z_3 \leq 0, &\lambda_1\\
  ~  &-4z_2 + z_3 \leq 0, &\lambda_2\\
  ~  &0 \leq z_1 \perp z_2 \geq 0 . &\sigma_1,\sigma_2
\end{aligned}
\end{equation*}
The global minimizer is $z^{*} = (z_1^{*},z_2^{*},z_3^{*})=(0,0,0)$ with $\beta(z^{*}) = 
\{1\}$. The multipliers at 
$z^{*}$ satisfy
\begin{equation*}
  \begin{aligned}
    \lambda_1 + \lambda_2 &= 1,\\
    \sigma_1 + \sigma_2 &= -2,
  \end{aligned}
\end{equation*}
and therefore $\sigma_1,\sigma_2$ cannot both be nonnegative.
Let $\beta_1 = \{1\}, \beta_2 = \emptyset$.
Since the problems NLP$_{(\beta_1,\beta_2)}$ and
NLP$_{(\beta_2,\beta_1)}$ have KKT multipliers with $(\sigma_1,\sigma_2) = (-2,0)$ and $(\sigma_1,\sigma_2)
= (0,-2)$, respectively, piecewise
M-stationarity holds at $z^{*}$.

Also, note that the KKT conditions of NLP$_{(\beta_1,\beta_2)}$ and
NLP$_{(\beta_2,\beta_1)}$ require that
\begin{subequations} \label{dual1}
\begin{align}
  \nabla f(\bar z) = \begin{bmatrix} 1\\1\\-1 \end{bmatrix}
  &\in \mathcal T_{(\beta_1,\beta_2)}^{\rm lin}(\bar z)^{*}
  = \left\{ w |
    \begin{bmatrix} w_1 \\ w_2 \\ w_3 \end{bmatrix}
    = \begin{bmatrix}
    4\lambda_1 + \sigma_1 \\
    4\lambda_2 + \sigma_2 \\
    -\lambda_1 - \lambda_2 \end{bmatrix},
    \lambda_1, \lambda_2, \sigma_2 \geq 0  \right\}, \\
  \nabla f(\bar z) = \begin{bmatrix} 1\\1\\-1 \end{bmatrix}
  &\in \mathcal T_{(\beta_2,\beta_1)}^{\rm lin}(\bar z)^{*}
  = \left\{ w |
    \begin{bmatrix} w_1 \\ w_2 \\ w_3 \end{bmatrix}
    = \begin{bmatrix}
    4\lambda_1 + \sigma_1 \\
    4\lambda_2 + \sigma_2 \\
    -\lambda_1 - \lambda_2 \end{bmatrix},
    \lambda_1, \lambda_2, \sigma_1 \geq 0 \right\},
\end{align}            
\end{subequations}
which relates $\nabla f(\bar z)$ to the branch NLPs.
Moreover, we have
\begin{equation*}
\begin{aligned}
  \mathcal T_{\rm MPCC}^{\rm lin}(\bar z)^{*}
  = \left\{ w \,|\, w^Td \geq 0,
    \,\forall d \in \mathcal T_{\rm MPCC}^{\rm lin}(\bar z) \right\}
  = \{ w \,|\, 
     w_1 \geq 0, \; w_2 \geq 0, \; w_3 \leq 0 \},
\end{aligned}
\end{equation*}
which validates $\nabla f(\bar z) \in \mathcal T_{\rm MPCC}^{\rm
  lin}(\bar z)^{*}$.

\subsubsection{M-stationary point} \label{eg:mstat}

This example shows piecewise M-stationarity failure at an M-stationary
point, which indicates that piecewise M-stationarity is a stronger
condition than M-stationarity. 
Consider the problem given by
\begin{equation*}
\begin{aligned}
  \min \quad &f(z) = (z_1-1)^2+z_2^2 &{\rm multipliers}\\
  {\rm s.t.} \quad &0 \leq z_1 \perp z_2 \geq 0. &\sigma_1,\sigma_2
\end{aligned}
\end{equation*}
The global minimizer is $z^{*} = (1,0)$, which is
S-stationary and trivially piecewise
M-stationary since $\beta(z^{*}) = \emptyset$.

Consider the point $\bar z=(0,0)$ with $\beta(\bar z) = \{1\}$.
The associated multipliers are
$(\sigma_1,\sigma_2) = (-2,0)$, and therefore $\bar z$ is an
M-stationary point of the MPCC and a KKT point
of the problem NLP$_{(\beta_1,\beta_2)}$ with $\beta_1 =
\{1\}, \beta_2 =\emptyset$. However,
$\bar z$ is not a KKT point of the problem
NLP$_{(\beta_2,\beta_1)}$,
indicating that stationarity fails for this branch problem and hence
$\bar z$ cannot be piecewise M-stationary.

\subsubsection{Fritz John point} \label{eg:fj}

This example shows piecewise M-stationarity failure at a minimizer
that is only a Fritz John
point for one of the branch problems. This example indicates that piecewise
M-stationarity may not hold at a minimizer that lacks an
appropriate CQ. 
Consider the following problem:
\begin{equation*}
\begin{aligned}
  \min \quad &f(z) = (z_1-1)^2 + (z_2+1)^2 \; &{\rm multipliers} \\
  {\rm s.t.} \quad &z_2^2 \leq 0, &\lambda\\
  ~ &0 \geq z_1 \perp z_2 \leq 0, &\sigma_1, \sigma_2
\end{aligned}
\end{equation*}
which searches for the minimal distance between points
$(z_1,z_2)$ and $(1,-1)$, along the negative axis of $z_1$.
The solution is $z^{*} = (0,0)$, where
the multipliers are
$(\sigma_1,\sigma_2)=(2,-2)$, and therefore $z^{*}$ is A-stationary.

Let $\beta_1=\{1\}, \beta_2=\emptyset$.
The only feasible solution for NLP$_{(\beta_1,\beta_2)}$ is
$z^{*} =(0,0)$, where NLP-GCQ fails (i.e., $\mathcal
T_{(\beta_1,\beta_2)}(\bar z)^{*} \neq \mathcal T_{(\beta_1,\beta_2)}^{\rm
lin}(\bar z)^{*}$) and KKT multipliers with $\sigma_2 \geq 0$ do not exist.
Instead, the Fritz John condition for this branch problem is given by
\begin{equation*}
  0 = \lambda_0 \begin{bmatrix} -2 \\ 2 \end{bmatrix}
  + \lambda \begin{bmatrix} 0 \\ 0 \end{bmatrix}
  + \sigma_1 \begin{bmatrix} 1 \\ 0 \end{bmatrix}
  + \sigma_2 \begin{bmatrix} 0 \\1 \end{bmatrix},
\end{equation*}
with $\lambda_0, \lambda, \sigma_2 \geq 0$ and $(\lambda_0,\lambda,\sigma_1,\sigma_2) \neq 0$.
This implies that $\lambda_0 = \sigma_1 = \sigma_2 = 0$ and
$\lambda>0$ can be chosen arbitrarily.
As a result, $z^{*}$ is only a Fritz John point on this branch
and cannot be piecewise M-stationary.

\subsection{Equivalence to B-stationarity}

This section establishes the
equivalence between piecewise M-stationarity (\ref{pw-mstat})  and
B-stationarity (\ref{bstat-d}) under an appropriate condition.  

\subsubsection{Piecewise NLP-GCQ and MPCC-GCQ}
We observe from the example in Section \ref{eg:fj} that piecewise
M-stationarity is valid only under appropriate conditions.
To discuss this property relying on the branch problems, this
section shows that MPCC-GCQ is an appropriate constraint
qualification.  

Consider the branch problems (\ref{b1b2}) at a feasible point $\bar z$
of an MPCC. 
The constraint set of a branch problem is called
\emph{Lagrange regular} at $\bar z$, if and only if NLP-GCQ holds at
$\bar z$ for this problem (see Theorem in \cite{gcq}).
Whenever the constraint set of \emph{every} branch problem is
Lagrange regular at $\bar z$, we say that \emph{piecewise}
NLP-GCQ holds at $\bar z$, namely,
\begin{equation} \label{pw-gcq}
  \mathcal T_{(\beta_1,\beta_2)}(\bar z)^{*} =
  \mathcal T_{(\beta_1,\beta_2)}^{\rm lin}(\bar z)^{*}, \quad
  \forall (\beta_1,\beta_2) \in \mathcal P(\beta(\bar z)).
\end{equation}

\begin{lem} \label{pw-gcq=mpcc-gcq}
Assume that $\bar z$ is a feasible point of MPCC (\ref{mpcc}) and
satisfies the B-stationarity condition (\ref{bstat-d}). Then,
piecewise NLP-GCQ holds at $\bar z$ if and only if MPCC-GCQ holds at
$\bar z$.
\end{lem}

\begin{proof}
We prove this lemma by showing that both piecewise NLP-GCQ and MPCC-GCQ
convert B-stationarity (\ref{bstat-d}) equivalently to the following
condition: 
\begin{equation} \label{bstat-pw}
  \nabla f(\bar z) \in 
  \bigcap_{(\beta_1,\beta_2) \in \mathcal P(\beta(\bar z))}
  \mathcal T_{(\beta_1,\beta_2)}^{\rm lin}(\bar z)^{*} .
\end{equation}

The condition (\ref{bstat-pw}) is the B-stationarity
condition (\ref{bstat-d}) under MPCC-GCQ by noting that
\begin{equation*}
\mathcal T_{\rm MPCC}^{\rm lin}(\bar z)^{*} = 
\left(
\bigcup_{(\beta_1,\beta_2) \in \mathcal P(\beta(\bar z))}
\mathcal T_{(\beta_1,\beta_2)}^{\rm lin}(\bar z)
\right)^{*}
 =  \bigcap_{(\beta_1,\beta_2) \in \mathcal P(\beta(\bar z))} \mathcal T_{(\beta_1,\beta_2)}^{\rm lin}(\bar z)^{*}. 
\end{equation*}
On the other hand, assume that piecewise NLP-GCQ holds at $\bar z$.
We express (\ref{bstat-d}) as
\begin{equation*}
  \nabla f(\bar z) \in 
  \bigcap_{(\beta_1,\beta_2) \in \mathcal P(\beta(\bar z))}
  \mathcal T_{(\beta_1,\beta_2)}(\bar z)^{*} ,
\end{equation*}
which is the same as the condition (\ref{bstat-pw}) because of
(\ref{pw-gcq}). 
This completes the proof.  
\end{proof}

\subsubsection{B-stationarity implies piecewise M-stationarity}
\begin{thm} \label{thm:nec}
Let $\bar z$ be a feasible point of MPCC (\ref{mpcc}) at which
B-stationarity (\ref{bstat-d}) and MPCC-GCQ hold. Then $\bar z$ is
piecewise M-stationary.    
\end{thm}

\begin{proof}
According to Lemma \ref{pw-gcq=mpcc-gcq}, piecewise NLP-GCQ holds at
$\bar z$ so that the B-stationarity condition
(\ref{bstat-d}) is equivalent to (\ref{bstat-pw}).
It follows  that $\bar z$ is a KKT point
of every branch problem NLP$_{(\beta_1,\beta_2)}$ \cite[Lemma 4.2]{gcq}. In particular, for
every $(\beta_1,\beta_2) \in \mathcal P(\beta(\bar z))$, there exist
$\eta_I^g, \eta_{\beta_1}^H, \eta_{\beta_2}^G \geq 0$ such that
\begin{equation} \label{kkt-b1b2}
\begin{aligned}
  0 &= &&\nabla f(\bar z)
     + \nabla g_I(\bar z) \eta_I^g + \nabla h(\bar z) \eta^h 
     - \nabla G_{\alpha}(\bar z) \eta_{\alpha}^G
     - \nabla H_{\gamma}(\bar z) \eta_{\gamma}^H \\
 ~ & ~
 &&- \nabla G_{\beta_1}(\bar z) \eta_{\beta_1}^G
     - \nabla H_{\beta_1}(\bar z) \eta_{\beta_1}^H 
     - \nabla G_{\beta_2}(\bar z) \eta_{\beta_2}^G
     - \nabla H_{\beta_2}(\bar z) \eta_{\beta_2}^H .
\end{aligned}
\end{equation}
On the other hand, the B-stationarity condition
(\ref{bstat-d}) under MPCC-GCQ is also
written as (\ref{bstat-pw}) and therefore we have
\begin{equation*}\label{bstat-lin}
  \nabla f(\bar z)^T d \geq 0, \quad
  \forall d \in \mathcal T_{\rm MPCC}^{\rm lin}(\bar z).
\end{equation*}
It follows from 
\cite[Theorem 6.1.1]{clarke83}, \cite[Lemma 1 and proof]{org} that
\begin{equation*} \label{fjconv}
\begin{aligned}
  0 &\in 
  &&\lambda_0\nabla f(\bar z)
  + \nabla g_I(\bar z)\lambda_I^g
  +\nabla h(\bar z)\lambda^h
  -\nabla G_{\alpha}(\bar z)\lambda_{\alpha}^G
  -\nabla H_{\gamma}(\bar z)\lambda_{\gamma}^H \\
  ~ & ~
  &&-\sum_{i \in \beta(\bar z)}\zeta_i \,
  {\rm conv} \{\nabla G_i(\bar z),\nabla H_i(\bar z)\},
\end{aligned}
\end{equation*}
where $\lambda_0 \geq 0, \lambda_I^g \geq 0$, and
$(\lambda_0,\lambda_I^g,\lambda^h,\lambda_{\alpha}^G,\lambda_{\gamma}^H, \zeta) \neq 0$;
conv$\{\nabla
G_i(\bar z),\nabla H_i(\bar z)\}$ represents the convex hull
consisting of all convex combinations of $\nabla G_i(\bar z)$ and
$\nabla H_i(\bar z)$.
Note that for every $i \in \beta(\bar z)$, 
 $\nabla G_i(\bar z)$ and $\nabla
H_i(\bar z)$ do not act on the condition
independently; instead, they are associated with a
common multiplier $\zeta_i$.
Let $\theta_i\nabla G_i(\bar z)+(1-\theta_i)\nabla
H_i(\bar z)$ with some $\theta_i \in [0,1]$ be an element of
the convex hull,
then we have
\begin{equation} \label{fj}
\begin{aligned}
  0 &=
  &&\lambda_0\nabla f(\bar z)
  + \nabla g_I(\bar z)\lambda_I^g
  +\nabla h(\bar z)\lambda^h
  -\nabla G_{\alpha}(\bar z)\lambda_{\alpha}^G
  -\nabla H_{\gamma}(\bar z)\lambda_{\gamma}^H \\
  ~ & ~
  &&-\sum_{i \in \beta_1}
  \underbrace{\zeta_i\theta_i}_{\lambda_i^G} \nabla G_i(\bar z)
  -\sum_{i \in \beta_1}
  \underbrace{\zeta_i(1-\theta_i)}_{\lambda_i^H} \nabla H_i(\bar z) \\
  ~ &~
  &&-\sum_{i \in \beta_2}
  \underbrace{\zeta_i\theta_i}_{\lambda_i^G} \nabla G_i(\bar z)
  -\sum_{i \in \beta_2}
  \underbrace{\zeta_i(1-\theta_i)}_{\lambda_i^H} \nabla H_i(\bar z). 
\end{aligned}
\end{equation}

Combining (\ref{kkt-b1b2}) and (\ref{fj}) reveals that
the relation (\ref{fj}) holds with $\lambda_0 >0$ and
$\lambda_I^g,\lambda_{\beta_1}^H,\lambda_{\beta_2}^G \geq 0$
for every $(\beta_1,\beta_2) \in \mathcal
P(\beta(\bar z))$.
To be specific, for every $(\beta_1,\beta_2) \in \mathcal
P(\beta(\bar z))$, letting the multipliers in
(\ref{kkt-b1b2}) take the form as in (\ref{fj}),
we can take $\lambda_0 =1$
(without loss of generality) and obtain
\begin{equation*}
\begin{gathered}
  \begin{matrix}
  \begin{array}{r@{}l}
    \lambda_I^g \, &= \, \eta_I^g \geq 0, \\
    \lambda_{\alpha}^G \, &= \,\eta_{\alpha}^G,
  \end{array}
  &
  \begin{array}{r@{}l}
    \lambda^h \, &= \, \eta^h, \\
   \lambda_{\gamma}^H \, &= \, \eta_{\gamma}^H,
  \end{array}
  \end{matrix}\\
  \begin{aligned}
  &\lambda_i^G =\zeta_i\theta_i = \eta_i^G,
  &&\lambda_i^H =\zeta_i(1-\theta_i) = \eta_i^H \geq 0, 
  &&\forall i \in \beta_1,\\
  &\lambda_i^G =\zeta_i\theta_i = \eta_i^G \geq 0, 
  &&\lambda_i^H =\zeta_i(1-\theta_i) = \eta_i^H,
  &&\forall i \in \beta_2.
  \end{aligned}
\end{gathered}
\end{equation*}
It follows from $\lambda_{\beta_1}^H,\lambda_{\beta_2}^G \geq 0$ that
\begin{equation} \label{pm-multiplier}
\begin{aligned}
  i &\in \beta_1
  \left\{ \begin{aligned}
      &\theta_i \in [0,1], \;
      \lambda_i^G \geq 0, \;\lambda_i^H \geq 0,
      &&\text{if } \zeta_i \geq 0;\\
      &\theta_i =1, \;
      \lambda_i^G =\zeta_i, \; \lambda_i^H = 0,
      &&\text{if } \zeta_i < 0;
  \end{aligned}\right.
  \\
  i &\in \beta_2
  \left\{ \begin{aligned}
      &\theta_i \in [0,1], \;
      \lambda_i^G \geq 0, \;\lambda_i^H \geq 0,
      &&\text{if } \zeta_i \geq 0;\\
      &\theta_i =0, \;
      \lambda_i^G = 0, \;\lambda_i^H =\zeta_i,
      &&\text{if } \zeta_i < 0.
    \end{aligned}\right.
\end{aligned}
\end{equation}
This shows that either $\lambda_i^G,\lambda_i^H
\geq 0$ or $\lambda_i^G\lambda_i^H=0$ for every $i \in \beta_1 \cup
\beta_2$ and every $(\beta_1,\beta_2) \in \mathcal
P(\beta(\bar z))$.
This completes the proof.
\end{proof}

\subsubsection{Piecewise M-stationarity implies B-stationarity}

\begin{thm} \label{thm:suf}
  Let $\bar z$ be a feasible point of MPCC (\ref{mpcc}) at which 
  piecewise M-stationarity (\ref{pw-mstat}) holds. Then $\bar z$ is
  B-stationary.
\end{thm}

\begin{proof}
From the piecewise M-stationarity condition (\ref{pw-mstat}),
we know that $\bar z$ is a KKT point of every branch problem
NLP$_{(\beta_1,\beta_2)}$ and we have from \cite[Lemma 4.2]{gcq} that
\begin{equation*}
\nabla f(\bar z) \in \mathcal T_{(\beta_1,\beta_2)}^{\rm lin}(\bar
z)^{*},
\quad\forall (\beta_1,\beta_2) \in \mathcal P(\beta(\bar z)).
\end{equation*}
It follows that
\begin{equation*}
  \nabla f(\bar z)^T d \geq 0, \quad
  \forall d \in 
  \bigcup_{(\beta_1,\beta_2) \in \mathcal P(\beta(\bar z))}
  \mathcal T_{(\beta_1,\beta_2)}^{\rm lin}(\bar z) ,
\end{equation*}
which is sufficient for
$\bar z$ to satisfy the B-stationarity conditions (\ref{bstat}) and (\ref{bstat-d}) because 
\begin{equation*}
\bigcup_{(\beta_1,\beta_2) \in \mathcal P(\beta(\bar z))}
\mathcal T_{(\beta_1,\beta_2)}^{\rm lin}(\bar z)
= \mathcal T_{\rm MPCC}^{\rm lin}(\bar z)
\supseteq \mathcal T(\bar z).
\end{equation*}
\end{proof}

\subsubsection{Equivalence}

\begin{cor} \label{cor:nec-suf}
  Let $\bar z$ be a feasible point of MPCC (\ref{mpcc}) at which
  MPCC-GCQ holds.
  Then $\bar z$ is B-stationary if
  and only if $\bar z$ is piecewise M-stationary. 
\end{cor}

\begin{proof}
The necessary part and sufficient part follow from Theorem \ref{thm:nec}
and Theorem \ref{thm:suf}, respectively.
\end{proof}

\subsection{Distinction from Extended M-stationarity}
So far we have shown that, under MPCC-GCQ, piecewise
M-stationarity is an equivalent concept to B-stationarity and
therefore a stronger concept than standard M-stationarity (\ref{mstat-st}).
A concept that has strong similarities to 
piecewise M-stationarity is \emph{extended M-stationarity}, which was
proposed by Gfrerer 
\cite{extended-mstat} and also discussed in Kim et al. \cite{bilevel-opt}.
Extended M-stationarity and its equivalence
to B-stationarity were investigated in
the background of disjunctive programs (problems in a more general
form than MPCCs), based on the concepts of generalized differentiation
(e.g., metric regularity, subregularity,
mixed regularity/subregularity, and their directional versions).
By contrast, here we have presented an elementary proof to establish the
equivalence between piecewise M-stationarity and B-stationarity,
without requiring consideration of
generalized differentiation concepts.

\section{Simplified B-Stationarity Verification} \label{sec:verify}
Based on the concept of piecewise M-stationarity, this section derives a simplified linear program with complementarity
constraints (LPCC) for B-stationarity verfication. 

Piecewise M-stationarity may reduce the cost of verifying
B-stationarity at a point that is not S-stationary.
Under MPCC-GCQ, the B-stationarity condition
\begin{equation*}
\nabla f(\bar z) \in \mathcal T_{\rm MPCC}^{\rm lin}(\bar z)^{*}
\end{equation*}
is equivalent to
having $d= 0$ as a global minimum of the following LPCC:
\begin{equation}\label{lpcc}
\begin{aligned}
  \min \quad &\nabla f(\bar z)^T d \\
  {\rm s.t.} \quad
    &\nabla g_i(\bar z)^T d \leq 0, &&i \in I_g(\bar z),\\
  ~ &\nabla h_i(\bar z)^T d = 0, &&i \in I_h(\bar z),\\
  ~ &\nabla G_i(\bar z)^T d = 0, &&i \in \alpha(\bar z),\\
  ~ &\nabla H_i(\bar z)^T d = 0, &&i \in \gamma(\bar z),\\
  ~ &0 \leq \nabla G_i(\bar z)^T d \perp 
     \nabla H_i(\bar z)^T d \geq 0, &&i \in \beta(\bar z).
\end{aligned}
\end{equation}
If $\bar z$ is B-stationary but not S-stationary, we know from
its piecewise M-stationarity, i.e., (\ref{pw-mstat}), that there exists an index set $\Omega \subseteq
\beta(\bar z)$ such that
\begin{subequations} \label{omega1}
\begin{gather}
  \bar\lambda_i^G, \bar\lambda_i^H \geq 0,
  \quad\forall i \in \beta(\bar z) \setminus \Omega,
  \label{smult}\\
  \left.
  \begin{aligned}
  &\bar\lambda_i^G <0, \; \bar\lambda_i^H =0,
  &&\forall i \in \beta_1 \cap \Omega, \\
  &\bar\lambda_i^G =0, \; \bar\lambda_i^H <0,
  &&\forall i \in \beta_2 \cap \Omega,
  \end{aligned}\right\}
  \forall (\beta_1,\beta_2) \in \mathcal P(\beta(\bar z)).
  \label{mmult}
\end{gather}
\end{subequations}
The multipliers (\ref{smult}) indicate that using
the following inequalities to replace the corresponding
complementarity constraints would not influence the global optimality
of $d = 0$ for LPCC (\ref{lpcc}):
\begin{equation*}
\nabla G_i(\bar z)^T d \geq 0, 
\nabla H_i(\bar z)^T d \geq 0,
\quad i \in \beta(\bar z) \setminus \Omega.
\end{equation*}
Therefore, piecewise M-stationarity will verify B-stationarity
based on the following simplified LPCC:
\begin{equation}\label{lpcc-sim}
\begin{aligned}
  \min \quad &\nabla f(\bar z)^T d \\
  {\rm s.t.} \quad
    &\nabla g_i(\bar z)^T d \leq 0, &&i \in I_g(\bar z),\\
  ~ &\nabla h_i(\bar z)^T d = 0, &&i \in I_h(\bar z),\\
  ~ &\nabla G_i(\bar z)^T d = 0, &&i \in \alpha(\bar z),\\
  ~ &\nabla H_i(\bar z)^T d = 0, &&i \in \gamma(\bar z),\\
  ~ &\nabla G_i(\bar z)^T d \geq 0,
    &&i \in \beta(\bar z) \setminus \Omega,\\
  ~ &\nabla H_i(\bar z)^T d \geq 0,
    &&i \in \beta(\bar z) \setminus \Omega, \\
  ~ &0 \leq \nabla G_i(\bar z)^T d \perp
     \nabla H_i(\bar z)^T d \geq 0,
    &&i \in \Omega.
\end{aligned}
\end{equation}
Piecewise M-stationarity of $\bar z$ for (\ref{lpcc-sim}) requires that
\begin{subequations} \label{omega2}
\begin{gather}
  \bar\lambda_i^G, \bar\lambda_i^H \geq 0,
  \quad\forall i \in \beta(\bar z) \setminus \Omega, \\
  \left.
  \begin{aligned}
  &\bar\lambda_i^G <0, \; \bar\lambda_i^H =0,
  &&\forall i \in \hat\beta_1, \\
  &\bar\lambda_i^G =0, \; \bar\lambda_i^H <0,
  &&\forall i \in \hat\beta_2,
  \end{aligned}\right\}
  \forall (\hat\beta_1,\hat\beta_2) \in \mathcal P(\Omega).
\end{gather}
\end{subequations}
Direct inspection reveals that (\ref{omega1}) and (\ref{omega2})
are essentially the same requirements on the multipliers. 
Using the simplified LPCC (\ref{lpcc-sim}) instead of the LPCC
(\ref{lpcc}) may reduce the cost of B-stationarity verification
significantly, when the cardinality of sets satisfy $|\Omega| \ll
|\beta(\bar z)|$.

The remaining question is how to determine the set $\Omega$. Note 
 that (\ref{smult}) also holds for the following
tightened linear program (TLP):
\begin{equation} \label{tlp}
\begin{aligned}
  \min \quad &\nabla f(\bar z)^T d \\
  {\rm s.t.} \quad
    &\nabla g_i(\bar z)^T d \leq 0, &&i \in I_g(\bar z),\\
  ~ &\nabla h_i(\bar z)^T d = 0, &&i \in I_h(\bar z),\\
  ~ &\nabla G_i(\bar z)^T d = 0, &&i \in \alpha(\bar z),\\
  ~ &\nabla H_i(\bar z)^T d = 0, &&i \in \gamma(\bar z),\\
  ~ &\nabla G_i(\bar z)^T d =0, \,
     \nabla H_i(\bar z)^T d =0, &&i \in \beta(\bar z).
\end{aligned}
\end{equation}
Suppose that (\ref{tlp}) is solved and the multipliers are obtained at
the solution. 
Then, the multipliers $\lambda_i^G, \lambda_i^H \geq 0$
identify elements in the set $i \in \beta(\bar z) \setminus \Omega$.
However, we may not need to solve (\ref{tlp}) directly, because
those nonnegative TLP multipliers can often be 
recognized from the multipliers at a solution obtained from a
NLP-based MPCC strategy, provided that the solution is at least weakly
stationary.

\section{Convergence of the Bounding Method}
\label{sec:ncp}

The results of Sections \ref{sec:nec} and \ref{sec:verify} are
independent of algorithms designed for solving MPCCs.
In the sequel, we investigate convergence properties of the 
bounding method we
proposed in \cite{bounding}.

\subsection{Brief Review of a Bounding Scheme}

In \cite{bounding}, we proposed an algorithm to seek a solution of MPCC
(\ref{mpcc}) by solving a sequence of NLP problems of the form
\begin{equation} \label{ba}
\begin{aligned}
{\rm BA}(\epsilon): \quad
  \min \quad &f(z) &{\rm multipliers}\\
{\rm s.t.} \quad &g(z) \leq 0, &u^g\\
  ~ &h(z) = 0, &u^h\\
  ~ &\Phi^{\epsilon}_i(z) + p_i = 0, \; i = 1, \dots, m, &u^{\Phi}_i
\end{aligned}
\end{equation}
where
\begin{equation} \label{phi}
\Phi^{\epsilon}_i(z) = \frac{1}{2}\left( G_i(z)+H_i(z) -
  \sqrt{(G_i(z)-H_i(z))^2+\epsilon^2} \right)
\end{equation}
is a smoothed NCP function (due to \cite{refncp}), with a
smoothing factor $\epsilon >0$, and the 
parameter $p_i$ adjusted adaptively (to 
take a value of zero or $\epsilon/2$).
We define the Lagrangian for the problem BA($\epsilon$) as
\begin{equation} \label{lag}
  \mathcal L_{\rm BA}(z,u) = f(z) + \sum_{i \in I_g(z)} u_{i}^g g_i(z) +
  \sum_{i \in I_h(z)} u_i^h h_i(z) -
  \sum_{i=1}^m u_i^{\Phi} (\Phi_i^{\epsilon}(z)+p_i) .
\end{equation}
As $\epsilon \to 0$,
a sequence of KKT points of BA($\epsilon$) tends
to a limit point. From \cite{bounding} the main results of this method are summarized below.
\begin{itemize}
\item
\emph{Feasibility:}
The smoothed NCP function (\ref{phi}) is used to approximate the
complementarity constraints in MPCC (\ref{mpcc}), and the largest
difference between 
them is $\epsilon/2$ \cite[Proposition 1.7]{bounding}.
When $\epsilon >0$, every feasible
point $z$ of BA($\epsilon$) satisfies
\begin{equation} \label{ncp}
\begin{aligned}
  &\Phi^{\epsilon}_i(z) + p_i = 0 \quad \Leftrightarrow \\
  &G_i(z)+p_i > 0,\;
   H_i(z)+p_i > 0,\;
   (G_i(z)+p_i)(H_i(z)+p_i) = (\epsilon/2)^2,
\end{aligned}
\end{equation}
whose limit at $\epsilon=0$ (thus $p_i=0$) recovers 
 $0 \leq G_i(z) \perp H_i(z) \geq 0$.

\item 
\emph{Sensitivity:}
At a KKT point $z(p)$ of BA($\epsilon$), these sensitivities
$\tfrac{{\rm d}f(z(p))}{{\rm d} p_i}$ are given by
$-u^{\Phi}_i$ for $i = 1, \dots, m$, provided that NLP-LICQ,
second-order 
sufficient conditions, and strict complementarity hold at $z(p)$.
This observation throws some light on the design of the Bounding
Algorithm, which takes advantage of the sensitivities to yield an 
efficient isolation of a solution to the MPCC. 

\item
\emph{Convergence:} 
The following convergence results have been
established for the Bounding Algorithm applied to
BA($\epsilon$).
\begin{enumerate}[label=(\roman*)]
\item
Suppose that MPCC-LICQ holds at a feasible point of the MPCC.
Then in a neighborhood of this point, NLP-LICQ holds at every
feasible point of BA($\epsilon$), whenever $\epsilon >0$ is
sufficiently small \cite[Theorem 3.1]{bounding}.

\item Suppose that a sequence of KKT points of problems
BA($\epsilon$) tends to a limit point as $\epsilon \to 0$, at which
MPCC-LICQ holds. Then the limit point is C-stationary
\cite[Theorem 3.3]{bounding}.

\item In addition, suppose that the reduced Hessian of the Lagrangian at
each of the KKT points of problems BA($\epsilon$) is
bounded below when $\epsilon >0$ 
is sufficiently small. Then the limit point is M-stationary
\cite[Theorem 3.5]{bounding}.
\end{enumerate}

\end{itemize}

\subsection{Bounding Algorithm}

The main idea of the Bounding Algorithm is given below.

Since $-\epsilon/2 \leq \Phi_i^{\epsilon}(z) \leq 0$ encloses the complementarity condition $0 \leq G_i(z) \perp H_i(z) \geq 0$ \cite[Proposition 1.7]{bounding}, 
we can approximate the complementarity with the equation $\Phi_i^{\epsilon}(z) + p_i = 0$ by letting $p_i \in [0, \epsilon/2]$ and $\epsilon \to 0$. 
For any parameters $p_i, p_i^{\prime} \in
[0,\epsilon/2]$ with $\epsilon >0$,
and the corresponding solutions $z(p)$ and
$z(p^{\prime})$ to BA($\epsilon$), it is straightforward to show that
\begin{equation*}
  f(z(p^{\prime})) = f(z(p)) + \left[ \frac{{\rm d}f(z(p))}{{\rm d}p}
  \right]^T (p^{\prime}-p) +
  O(\|p^{\prime}-p\|^2) .
\end{equation*}
Noting that the sensitivities $\tfrac{{\rm d}f(z(p))}{{\rm d} p}$ are given
by $-u^{\Phi}$, we have that
\begin{equation*}
  f(z(p)) - \frac{\epsilon}{2} \sum_{i = 1}^m |u_i^{\Phi}(p)|
  - |O(\epsilon^2)|
  \leq f(z(p^{\prime})) \leq 
  f(z(p)) + \frac{\epsilon}{2} \sum_{i = 1}^m |u_i^{\Phi}(p)|
  + |O(\epsilon^2)| .
\end{equation*}
This relation describes the estimates of the objective by the 
Bounding Algorithm shown below.
In order to seek a solution to the MPCC,
the problem BA($\epsilon^k$) is solved
to a local solution $z^k$, for a sequence of smoothing parameters
$\{\epsilon^k\} \to 0$. At each solution $z^k = z(p^k)$,
we take advantage of the
sensitivities $u^{\Phi,k}=u^{\Phi}(p^k)$ to adjust the parameters
$p^k$, with the aim 
of improving the objective in the subsequent solution.
As shown in Step \ref{step:sen} of the algorithm, 
in the case of $u_i^{\Phi,k} >0$ with $i \in P_0$ ($u_i^{\Phi,k} <0$
with $i \in P_{\epsilon}$),
increasing (decreasing) $p_i^k$ may lead to an improvement 
in the objective function.
Step \ref{step:param} then adjusts the parameters for the next solution;
$p_i^{k+1}$ takes the extreme values for $i \in P_0 \cup P_{\epsilon}$. 
Otherwise, $p_i^{k+1}$ is set to a fraction of $p_i^k$ for $i \notin P_0 \cup P_{\epsilon}$.
When $\epsilon^k >0$ is sufficiently small, 
$z(p^k)$ is an $\epsilon^k$-approximate solution to the MPCC,
which includes 
an $O((\epsilon^k)^2)$ correction arising from the
parameter adjustments \cite[Proposition 2.1]{bounding}. 

\renewcommand{\thealgorithm}{} 
\begin{algorithm}
\caption{A Bounding Algorithm for MPCCs}
\begin{algorithmic}[0] 
\setlength{\itemsep}{5pt} 
\STATE \textbf{Initialization:} 
Specify initial smoothing factor $\epsilon^0 >
0$, reducing factor $\kappa \in (0, 1)$, initial point $z^0$, and solution
tolerance $\epsilon_{\rm tol} > 0$. Set initial parameters $p^0
\leftarrow 0$, and counter $k \leftarrow 0$.

\STATE \textbf{Main loop:}
While $\epsilon^k \ge \epsilon_{\rm tol}$, do the following.

\begin{enumerate}
\item \label{step:nlp}
Solve the problem BA($\epsilon^k$). With parameters $p^k$, solve the
problem BA($\epsilon^k$) to obtain a
stationary point $z^k$ and multipliers $u^k = (u^{g,k}, u^{h,k}, u^{\Phi,k})$.

\item \label{step:sen}
Estimate bounds of the MPCC’s objective. Define the index sets
\begin{equation*}
\begin{aligned}
  P_0 &= \{i \,|\, p_i^k = 0 \text{ and } 
         u^{\Phi,k}_i > 0\}, \\
  P_{\epsilon} &= \{i \,|\, p_i^k =\epsilon^k/2 \text{ and }
         u^{\Phi,k}_i <0 \}.
\end{aligned}
\end{equation*}
Then the following settings would reduce $f(z^k)$:
\begin{equation*}
\begin{aligned}
  p_i^k &\leftarrow \epsilon^k/2, &&\forall i \in P_0, \\
  p_i^k &\leftarrow 0, &&\forall i \in P_{\epsilon}.
\end{aligned}
\end{equation*}
The bounds of the MPCC's objective value can be estimated by
\begin{equation*}
\begin{aligned}
f^{\rm low} &= f(z^k) - \epsilon^k \sum_{i \in P_0 \cup P_{\epsilon}}
|u^{\Phi,k}_i|,  \\
f^{\rm up} &= f(z^k) + \epsilon^k \sum_{i \in P_0 \cup P_{\epsilon}} |u^{\Phi,k}_i|.
\end{aligned}
\end{equation*}

\item \label{step:param}
Update the parameters $\epsilon$ and $p$.
Set  $\epsilon^{k+1} \leftarrow \kappa \epsilon^k$, and
\begin{equation*}
  p_i^{k+1} =  \left\{
    \begin{aligned}
       &\epsilon^{k+1}/2, &&i \in P_0 , \\
       &0, &&i \in P_{\epsilon} , \\
       &\kappa p_i^k, &&{\rm otherwise} .
    \end{aligned} \right.
\end{equation*}

\item
  Set $k \leftarrow k+1$ and go to Step 1.
\end{enumerate}
\end{algorithmic}
\end{algorithm}

\subsection{Derivatives of  $\Phi^{\epsilon}$}

Derivatives of the function $\Phi^{\epsilon}$ defined in (\ref{phi})
are derived in \cite{bounding}.
The following gives the main results for their later use.

With $\epsilon >0$, the first and second derivatives of the function
$\Phi_i^{\epsilon}(z)$ in (\ref{phi}) are given by 
\begin{equation*}
\begin{aligned}
  \nabla_G \Phi^{\epsilon}_i(z) &= \frac{1}{2} -
    \frac{G_i(z)-H_i(z)}{2\sqrt{(G_i(z)-H_i(z))^2+\epsilon^2}}, \\
  \nabla_H \Phi^{\epsilon}_i(z) &= \frac{1}{2} +
    \frac{G_i(z)-H_i(z)}{2\sqrt{(G_i(z)-H_i(z))^2+\epsilon^2}}, \\
  \nabla_{GG} \Phi^{\epsilon}_i(z) = \nabla_{HH} \Phi^{\epsilon}_i(z) &= 
    \frac{-\epsilon^2}{2[(G_i(z)-H_i(z))^2+\epsilon^2]^{3/2}}, \\
  \nabla_{GH} \Phi^{\epsilon}_i(z) = \nabla_{HG} \Phi^{\epsilon}_i(z) &= 
    \frac{\epsilon^2}{2[(G_i(z)-H_i(z))^2+\epsilon^2]^{3/2}}.
\end{aligned}
\end{equation*}
Let $z$ satisfy $\Phi_i^{\epsilon}(z)+p_i = 0$ with $\epsilon >0$. It
follows from 
(\ref{ncp}) that
\begin{multline*}
  \sqrt{(G_i(z)-H_i(z))^2 +\epsilon^2}
  = \sqrt{((G_i(z)+p_i)-(H_i(z)+p_i))^2+\epsilon^2} \\
  \begin{aligned}
  &= \sqrt{(G_i(z)+p_i)^2+(H_i(z)+p_i)^2+2(G_i(z)+p_i)(H_i(z)+p_i)} \\
  &= |G_i(z)+H_i(z)+2p_i| = G_i(z) + H_i(z) + 2p_i.
  \end{aligned}
\end{multline*}
Using this and $(G_i(z)+p_i)(H_i(z)+p_i) = (\epsilon/2)^2$,
we can rephrase the above derivatives as
\begin{equation} \label{pdphi}
\begin{aligned}
  \nabla_G \Phi^{\epsilon}_i(z) &=  \frac{H_i(z)+p_i}{G_i(z)+H_i(z)+2p_i}, \\
  \nabla_H \Phi^{\epsilon}_i(z) &=  \frac{G_i(z)+p_i}{G_i(z)+H_i(z)+2p_i}, \\
  \nabla_{GG} \Phi^{\epsilon}_i(z) = \nabla_{HH} \Phi^{\epsilon}_i(z)
    &=  \frac{-2(G_i(z)+p_i)(H_i(z)+p_i)}{(G_i(z)+H_i(z)+2p_i)^3}, \\
  \nabla_{GH} \Phi^{\epsilon}_i(z) = \nabla_{HG} \Phi^{\epsilon}_i(z)
    &=  \frac{2(G_i(z)+p_i)(H_i(z)+p_i)}{(G_i(z)+H_i(z)+2p_i)^3}.
\end{aligned}
\end{equation}

\subsection{C-Stationarity Convergence} \label{sec:cstat}
Let a sequence $\{z^k\} \to \bar z$ as $\epsilon^k \to 0$, where every
$z^k$ is a KKT point of BA($\epsilon^k$), namely,
there exist multipliers 
$u^k = (u^{g,k}, u^{h,k}, u^{\Phi,k})$ with $u^{g,k} \geq 0$
such that the following first-order necessary condition is satisfied:
\begin{equation} \label{kktba}
  0 = \nabla f(z^k) 
      + \sum_{i \in I_g(z^k)} u_i^{g,k} \nabla g_i(z^k) 
      + \sum_{i \in I_h(z^k)} u_i^{h,k} \nabla h_i(z^k) 
      - \sum_{i =1}^{m} u^{\Phi,k}_i \nabla \Phi_i^{\epsilon}(z^k).
\end{equation}
The following
establishes C-stationarity of $\bar z$. 

\begin{thm} \label{thm:cstatba}
For a sequence of positive scalars $\epsilon^k \to 0$, apply the
Bounding Algorithm to BA($\epsilon^k$).
Assume this generates a sequence
$\{z^k\} \to \bar z$, where every $z^k$ is a KKT point of
BA($\epsilon^k$).
Then $\bar z$ is C-stationary.
\end{thm}

\begin{proof}
\emph{Derivatives in the limit.}
As $\epsilon^k$ tends to zero, the function
$\Phi_i^0$ is in general not  
differentiable for $i \in \beta(\bar z)$.
Noting that $\Phi_i^0(\bar z) = \min\{G_i(\bar z), H_i(\bar z)\} = 0$,
we have 
\begin{equation*}
  \partial \Phi_i^0(\bar z)
= \partial \min\{G_i(\bar z), H_i(\bar z)\}
  = {\rm conv} \{ \nabla G_i(\bar z), \nabla H_i(\bar z) \} .  
\end{equation*}
For $\delta_i \in \partial \Phi_i^0(\bar z)$, it follows from
\cite[proof of Lemma 1]{org} that
\begin{equation*}
\begin{split}
   \delta_i = \theta_i \nabla G_i(\bar z) +
   (1-\theta_i) \nabla H_i(\bar z),\\
   \theta_i G_i(\bar z) = 0, \\
   (1-\theta_i) H_i(\bar z) = 0, &\quad  \mbox{ for some } \theta_i \in [0,1].
\end{split}
\end{equation*}
Therefore, we have
\begin{equation} \label{dphi}
   \partial \Phi_i^0 =  \left\{
   \begin{aligned}
      &\nabla G_i(\bar z), &&i \in \alpha(\bar z), \\ 
      &\nabla H_i(\bar z), &&i \in \gamma(\bar z), \\ 
      &\{ \theta_i \nabla G_i(\bar z) + (1-\theta_i) \nabla H_i(\bar
      z) \},  &&i \in \beta(\bar z),
   \end{aligned} \right.
\end{equation}
where  $\theta_i \in [0,1]$.

\emph{Well-defined multipliers.}
Consider the relation (\ref{kktba}).
Without loss of generality, we have 
$\nabla f(z^k) \neq 0$ (otherwise $z^k$ is an unconstrained local
minimizer of BA($\epsilon^k$)). It follows from (\ref{kktba}) that $\nabla f(z^k)^T d \geq
0$ for all $d \in \mathcal T_{\rm BA}^{\rm lin}(z^k)$, and hence
$\nabla f(z^k)$ belongs to the cone $\mathcal T_{\rm BA}^{\rm lin}(z^k)^{*}$.
According to \emph{Caratheodory's theorem for cones} (see for example
\cite[Exercise B.1.7]{bertsekas-nlp} and \cite{cpld-is-cq}), there exist
\begin{equation*}
\begin{aligned}
  J_g^k &\subseteq \{i \,|\, u_i^{g,k} >0\} \subseteq I_g(z^k), \\
  J_h^k &\subseteq \{i \,|\, u_i^{h,k} \neq 0\} \subseteq I_h(z^k), \\
  J_{\Phi}^k &\subseteq \{i \,|\, u_i^{\Phi,k} \neq 0\} \subseteq \{1,\dots,m\},
\end{aligned}
\end{equation*}
such that the gradients
\begin{equation} \label{indep-zk}
\begin{aligned}
  \nabla g_i(z^k), \quad &i \in J_g^k, \\
  \nabla h_i(z^k), \quad &i \in J_h^k, \\
  \nabla \Phi_i^{\epsilon}(z^k)
  = \nabla_G \Phi_i^{\epsilon}(z^k) \nabla G_i(z^k) +
     \nabla_H \Phi_i^{\epsilon}(z^k) \nabla H_i(z^k), \quad &i \in J_{\Phi}^k
\end{aligned}
\end{equation}
are linearly independent and $\nabla f(z^k)$ can be represented by
\begin{equation} \label{kktba-indep}
  -\nabla f(z^k)  = 
        \sum_{i \in J_g^k} w_i^{g,k} \nabla g_i(z^k) 
      + \sum_{i \in J_h^k} w_i^{h,k} \nabla h_i(z^k) 
      - \sum_{i \in J_{\Phi}^k} w^{\Phi,k}_i \nabla \Phi_i^{\epsilon}(z^k),
\end{equation}
with
\begin{equation} \label{wk}
\begin{aligned}
  w_i^{g,k} &> 0,
  &&\forall i \in J_g^k, \\
  w_i^{h,k}u_i^{h,k} &>0,
  &&\forall i \in J_h^k, \\
  w_i^{\Phi,k}u_i^{\Phi,k} &>0, 
  &&\forall i \in J_{\Phi}^k.
\end{aligned}
\end{equation}
Note that $w^k=(w^{g,k},w^{h,k},w^{\Phi,k})$ are well-defined because
of the independence of the vectors (\ref{indep-zk}), and
the sequence $\{w^k\}$ must converge to some
limit $\bar w = (\bar w^g, \bar w^h, \bar w^{\Phi})$.
It follows from (\ref{kktba-indep}) that 
this limit must satisfy
\begin{equation} \label{dl-lim}
\begin{aligned}
  - \nabla f(\bar z) &\in
       \sum_{i \in \bar J_g} \bar w_i^g \nabla g_i(\bar z)
       + \sum_{i \in \bar J_h} \bar w_i^h \nabla h_i(\bar z) \\
      &- \sum_{i \in \bar J_{\Phi} \cap \alpha(\bar z)}
         \bar w^{\Phi}_i \nabla G_i(\bar z)
       - \sum_{i \in \bar J_{\Phi} \cap \gamma(\bar z)} 
         \bar w^{\Phi}_i \nabla H_i(\bar z)\\
      &- \sum_{i \in \bar J_{\Phi} \cap \beta(\bar z)}
         \{ \bar w^{\Phi}_i \theta_i\nabla G_i(\bar z) +
            \bar w^{\phi}_i (1-\theta_i)\nabla H_i(\bar z) \} ,
\end{aligned}
\end{equation}
where (\ref{dphi}) has been used to characterize the derivatives at
$\bar z$, and
\begin{equation*}
  \bar J_g \subseteq I_g(\bar z), \;
  \bar J_h \subseteq I_h(\bar z), \;
  \bar J_{\Phi} \subseteq \{1,\dots,m\}.
\end{equation*}

\emph{Weak and C- stationarity.}
It follows from (\ref{dl-lim}) that
\begin{equation} \label{dl-lim2}
\begin{aligned}
  0 &= \nabla f(\bar z)
       + \sum_{i \in \bar J_g} \bar w_i^g \nabla g_i(\bar z)
       + \sum_{i \in \bar J_h} \bar w_i^h \nabla h_i(\bar z) \\
      &- \sum_{i \in \bar J_{\Phi} \cap \alpha(\bar z)}
         \bar w^{\Phi}_i \nabla G_i(\bar z)
       - \sum_{i \in \bar J_{\Phi} \cap \gamma(\bar z)} 
         \bar w^{\Phi}_i \nabla H_i(\bar z) \\
      &- \sum_{i \in \bar J_{\Phi} \cap \beta(\bar z)}
         \bar w^{\Phi}_i [\theta_i\nabla G_i(\bar z) + (1-\theta_i)\nabla H_i(\bar z)] ,
\end{aligned}
\end{equation}
for some $\theta_i \in [0,1]$.
Thus $\bar z$ satisfies the weak stationarity condition,  
with the MPCC multipliers given by
\begin{equation} \label{wmultba}
\begin{aligned}
  \bar \lambda_i^g &= \left\{
    \begin{aligned}
      &\bar w_i^g = \lim_{k \to \infty} w_i^{g,k},
      &&\text{if } i \in \bar J_g, \\
      &0,
      &&\text{otherwise},
    \end{aligned}
    \right. \\
  \bar \lambda^h &= \left\{
    \begin{aligned}  
      &\bar w_i^h = \lim_{k \to \infty} w_i^{h,k},
      &&\text{if } i \in \bar J_h, \\
      &0,
      &&\text{otherwise},
    \end{aligned}
    \right. \\
  \bar \lambda_i^G &= \left\{
    \begin{aligned}
      &\bar w_i^{\Phi} = \lim_{k \to \infty} w_i^{\Phi,k},
      &&\text{if } i \in \bar J_{\Phi} \cap \alpha(\bar z), \\ 
      &\bar w_i^{\Phi} \theta_i, 
      &&\text{if } i \in \bar J_{\Phi} \cap \beta(\bar z), \\
      &0,
      &&\text{otherwise},
    \end{aligned}
  \right. \\
  \bar \lambda_i^H &= \left\{
    \begin{aligned}
      &\bar w^{\Phi}_i = \lim_{k \to \infty} w_i^{\Phi,k},
      &&\text{if } i \in \bar J_{\Phi} \cap \gamma(\bar z), \\
      &\bar w^{\Phi}_i (1-\theta_i),
      &&\text{if } i \in \bar J_{\Phi} \cap \beta(\bar z), \\
      &0,
      &&\text{otherwise},
    \end{aligned}
  \right.
\end{aligned}
\end{equation}
for some $\theta_i \in [0,1]$.
Moreover, $\bar z$ is C-stationary because
\begin{equation} \label{cmult}
  \bar \lambda^G_i \cdot \bar \lambda^H_i = \left\{ 
  \begin{aligned} 
    &(\bar w^{\Phi}_i)^2\theta_i(1-\theta_i) \geq 0,
    &&\text{if } i \in \bar J_{\Phi} \cap \beta(\bar z), \\
    &0,
    &&\text{if } i \in \beta(\bar z) \setminus \bar J_{\Phi}.
  \end{aligned}
\right.
\end{equation} 
\end{proof}

\subsection{Piecewise M-Stationarity Convergence} \label{sec:pw-mstat}

Consider the Lagrangian function
\begin{equation*}
  \mathcal{L}_{\rm BA}(z,w) = f(z) + \sum_{i \in J_g^k} w_{i}^g g_i(z) +
  \sum_{i \in J_h^k} w_i^h h_i(z) -
  \sum_{i \in J_{\Phi}^k} w_i^{\Phi} (\Phi_i^{\epsilon}(z)+p_i) .
\end{equation*}
The reduced Hessian of the Lagrangian is said to be bounded
below at $z^k$ if
\begin{equation} \label{rhess}
  d^T \nabla_{zz} \mathcal L_{\rm BA}(z^k,w^k) d > -\infty, \quad
  \forall d \in \mathcal D(z^k),
\end{equation}
where
\begin{equation*}
\begin{aligned}
    \mathcal D(z^k) = \{ d \,|\,
      &\nabla g_i(z^k)^Td = 0, &&\forall i \in J_g^k,\\
    ~ &\nabla h_i(z^k)^Td = 0, &&\forall i \in J_h^k,\\
    ~ &\nabla \Phi_i^{\epsilon}(z^k)^Td = 0,
      &&\forall i \in J_{\Phi}^k\}.
\end{aligned}
\end{equation*}

\begin{thm} \label{thm:mstatba}
Suppose that $\bar z$ is generated from the  sequence described in
Theorem \ref{thm:cstatba}. 
In addition to the assumptions of Theorem \ref{thm:cstatba},
suppose that the condition (\ref{rhess}) holds at every $z^k$ for 
$k$ sufficiently large. 
Then $\bar z$ is piecewise M-stationary and therefore B-stationary. 
\end{thm}

\begin{proof}
At a limit point $\bar z$, we know from the weak stationarity of $\bar
z$ that $\nabla f(\bar z)$ belongs to the cone $\mathcal
T_{\rm TNLP}^{\rm lin}(\bar z)^{*}$, where $\mathcal
T_{\rm TNLP}^{\rm lin}(\bar z)$ represents the cone of
first-order feasible variations of the following tightened NLP
at $\bar z$:
\begin{equation*}
\begin{aligned}
{\rm TNLP}: \quad \min \quad &f(z) \\
{\rm s.t.} \quad &g(z) \leq 0, \\
  ~ &h(z) = 0,\\
  ~ &G_i(z) = 0, &&i \in \alpha(\bar z),\\
  ~ &H_i(z) = 0, &&i \in \gamma(\bar z),\\
  ~ &G_i(z) = 0, \; H_i(z) = 0, &&i \in \beta(\bar z).
\end{aligned}
\end{equation*}
According to
\emph{Caratheodory's theorem for cones}, there exist
\begin{equation*}
  J_g^{*} \subseteq I_g(\bar z), \;
  J_h^{*} \subseteq I_h(\bar z), \;
  J_G^{*} \subseteq \alpha(\bar z) \cup \beta(\bar z), \;
  J_H^{*} \subseteq \gamma(\bar z) \cup \beta(\bar z),
\end{equation*}
such that the gradients
\begin{equation} \label{indep-lim}
\begin{aligned}
  \nabla g_i(z),\quad &i \in J_g^{*}, \\
  \nabla h_i(z),\quad &i \in J_h^{*}, \\
  \nabla G_i(z),\quad &i \in J_G^{*}, \\
  \nabla H_i(z),\quad &i \in J_H^{*}, 
\end{aligned}
\end{equation}
are linearly independent at $\bar z$ and $\nabla f(\bar z)$ can be represented by
\begin{equation*} \label{kktba-indep-lim}
  -\nabla f(\bar z)  = 
        \sum_{i \in J_g^{*}} \bar \lambda_i^g \nabla g_i(\bar z) 
      + \sum_{i \in J_h^{*}} \bar \lambda_i^h \nabla h_i(\bar z) 
      - \sum_{i \in J_G^{*}} \bar \lambda_i^G \nabla G_i(\bar z)
      - \sum_{i \in J_H^{*}} \bar \lambda_i^H \nabla H_i(\bar z),
\end{equation*}
with
\begin{equation*}
\begin{aligned}
  \bar \lambda_i^g &> 0, &&\forall i \in J_g^{*}, \\
  \bar \lambda_i^h &\neq 0, &&\forall i \in J_h^{*}, \\
  \bar \lambda_i^G &\neq 0, &&\forall i \in J_G^{*}, \\
  \bar \lambda_i^H &\neq 0, &&\forall i \in J_H^{*}, \\
  (\bar\lambda_i^g, \bar\lambda_i^h, \bar\lambda_i^G, \bar\lambda_i^H) &= 0,
  &&\forall i \notin J_g^{*} \cup J_h^{*} \cup J_G^{*} \cup J_H^{*}.
\end{aligned}
\end{equation*}

For $k$ sufficiently large, the gradients in (\ref{indep-lim})
are also linearly independent at $z^k$ because of their
continuity. Therefore, we may assume that there exists a sequence
$\{z^k\} \to \bar z$, with (\ref{indep-zk})--(\ref{wk}) holding with
\begin{equation*}
  J_g^k \subseteq J_g^{*}, \;
  J_h^k \subseteq J_h^{*}, \;
  J_{\Phi}^k \subseteq J_G^{*} \cup J_H^{*},
\end{equation*}
and (\ref{dl-lim2}) holding with
\begin{equation*}
  \bar J_g = J_g^{*}, \;
  \bar J_h = J_h^{*}, \;
  \bar J_{\Phi} = J_G^{*} \cup J_H^{*}.
\end{equation*}
Suppose that $\bar z$ is not piecewise M-stationary, in order to derive a contradiction.
Then, $\bar z$ is
C-stationary (Theorem \ref{thm:cstatba}) but there exists
a partition $(\hat\beta_1,\hat\beta_2) \in \mathcal P(\beta(\bar z))$
such that $\bar z$ is not a KKT point of the branch problem
NLP$_{(\hat\beta_1,\hat\beta_2)}$. 
It follows that
\begin{subequations} \label{hypo}
\begin{align} 
  \text{either } &\exists i_1 \in \hat\beta_1 \text{ such that }
  \bar\lambda_{i1}^H = \bar w_{i1}^{\Phi}(1-\theta_{i1}) <0, \label{caseb1}\\
  \text{or } &\exists i_2 \in \hat\beta_2 \text{ such that }
  \bar\lambda_{i2}^G = \bar w_{i2}^{\Phi}\theta_{i2} <0. \label{caseb2}
\end{align}
\end{subequations} 
This implies that
\begin{equation*}
\begin{aligned}
  \text{either } &\exists i_1 \in \hat\beta_1 \text{ such that }
  0<\theta_{i1}<1, \;
  \bar w_{i1}^{\Phi}, \bar\lambda_{i1}^G, \bar\lambda_{i1}^H <0, \\
  \text{or } &\exists i_2 \in \hat\beta_2 \text{ such that }
  0<\theta_{i2}<1, \; 
  \bar w_{i2}^{\Phi}, \bar\lambda_{i2}^G, \bar\lambda_{i2}^H <0.
\end{aligned}
\end{equation*}

The following shows that neither of the hypotheses in (\ref{hypo}) can
happen under the assumptions of the theorem.
The analysis is carried out with (\ref{caseb1}) and can be adapted to
(\ref{caseb2}) in a straightforward way.
For $k$ sufficiently large, we can choose a sequence of directions $\{d^{\hat\beta_1}\}$ with
\begin{equation} \label{dir1}
\begin{aligned}
  \nabla g_i(z^k)^T d^{\hat\beta_1} &= 0, \quad \forall i \in J_g^k,\\
  \nabla h_i(z^k)^T d^{\hat\beta_1} &= 0, \quad \forall i \in J_h^k,\\
  \nabla G_i(z^k)^T d^{\hat\beta_1} &= 0, \quad
  \forall i \in J_{\Phi}^k \cap J_G^{*} \setminus \{i_1\},\\
  \nabla H_i(z^k)^T d^{\hat\beta_1} &= 0, \quad
  \forall i \in  J_{\Phi}^k \cap J_H^{*} \setminus \{i_1\},\\
  \nabla G_{i_1}(z^k)^T d^{\hat\beta_1} &= \nabla_H \Phi^{\epsilon}_{i_1}(z^k),\\
  \nabla H_{i_1}(z^k)^T d^{\hat\beta_1} &= -\nabla_G \Phi^{\epsilon}_{i_1}(z^k).
\end{aligned}
\end{equation}
The sequence $\{d^{\hat\beta_1}\}$ can be bounded, because $\nabla_G
\Phi_{i_1}^{\epsilon}(z^k)+\nabla_H\Phi_{i_1}^{\epsilon}(z^k)=1$.  
Note that $d^{\hat\beta_1} \in \mathcal D(z^k)$.
Combining (\ref{dir1}) and the derivatives
in (\ref{pdphi}), we obtain the contribution of the constraint
$\Phi_{i_1}^{\epsilon}(z^k)+p_{i_1}^k = 0$ to the reduced
Hessian of the Lagrangian function $\mathcal L_{\rm BA}(z^k,w^k)$, that is,
\begin{equation*} \label{rhess-phi}
\begin{aligned}
  - &w_{i_1}^{\Phi,k} (d^{\hat\beta_1})^T\nabla_{zz}\Phi^{\epsilon}_{i_1}(z^k) d^{\hat\beta_1}\\
  = &-w_{i_1}^{\Phi,k} (d^{\hat\beta_1})^T \bigl[ 
   \nabla_G\Phi^{\epsilon}_{i_1}(z^k) \nabla_{zz} G_{i_1}(z^k) +
   \nabla_H\Phi^{\epsilon}_{i_1}(z^k) \nabla_{zz} H_{i_1}(z^k) \\
  ~ &+ \nabla_{GG}\Phi^{\epsilon}_{i_1}(z^k) \nabla G_{i_1}(z^k) \nabla G_{i_1}(z^k)^T +
   \nabla_{GH}\Phi^{\epsilon}_{i_1}(z^k) \nabla G_{i_1}(z^k) \nabla H_{i_1}(z^k)^T \\
  ~ &+ \nabla_{HG}\Phi^{\epsilon}_{i_1}(z^k) \nabla H_{i_1}(z^k) \nabla G_{i_1}(z^k)^T +
   \nabla_{HH}\Phi^{\epsilon}_{i_1}(z^k) \nabla H_{i_1}(z^k) \nabla H_{i_1}(z^k)^T
   \bigr] d^{\hat\beta_1} \\
  = &-w_{i_1}^{\Phi,k} (d^{\hat\beta_1})^T\nabla_G\Phi^{\epsilon}_{i_1}(z^k) \nabla_{zz}G_{i_1}(z^k)d^{\hat\beta_1}
   -w_{i_1}^{\Phi,k} (d^{\hat\beta_1})^T\nabla_H\Phi^{\epsilon}_{i_1}(z^k) \nabla_{zz}H_{i_1}(z^k)d^{\hat\beta_1}\\
  &+\frac{2 w_{i_1}^{\Phi,k}}{G_{i_1}(z^k)+H_{i_1}(z^k)+2p_{i_1}^k} 
   \nabla_G \Phi^{\epsilon}_{i_1}(z^k) \nabla_H \Phi^{\epsilon}_{i_1}(z^k).
\end{aligned}
\end{equation*}
In the last equality, the first two terms are bounded;
in the third term,  
$\nabla_G \Phi_{i_1}^{\epsilon}(z^k)$ and $\nabla_H \Phi_{i_1}^{\epsilon}(z^k)$
are positive and bounded,
$w_{i_1}^{\Phi,k}$ tends to $\bar w_{i_1}^{\Phi} <0$, while 
$G_{i_1}(z^k)$, $H_{i_1}(z^k)$, and $p_{i_1}^k$ are positive sequences
tending to zero.
As a result, we have
\begin{equation} \label{rhessphi}
- w_{i_1}^{\Phi,k} (d^{\hat\beta_1})^T\nabla_{zz}\Phi^{\epsilon}_{i_1}(z^k) d^{\hat\beta_1}
\to -\infty, \;\exists d^{\hat\beta_1} \in \mathcal D(z^k).
\end{equation}
Since all other terms in the reduced Hessian $(d^{\hat\beta_1})^T
\nabla_{zz}\mathcal L_{\rm BA}(z^k,w^k)d^{\hat\beta_1}$ are bounded,
we obtain that (\ref{rhessphi}) contradicts the condition (\ref{rhess}).
Hence, the hypothesis (\ref{caseb1}) must be false.
In an analogous way, we can construct a direction $d^{\hat\beta_2} \in
\mathcal D(z^k)$ to show that the hypothesis (\ref{caseb2}) must be false too.

Since $\bar z$ is both C-stationary and a KKT point of
NLP$_{(\beta_1,\beta_2)}$ for every $(\beta_1,\beta_2) \in \mathcal
P(\beta(\bar z))$, then there exist multipliers
$\bar\lambda=(\bar\lambda^g, \bar\lambda^h, \bar\lambda^G,
\bar\lambda^H)$ at $\bar z$ such that
\begin{subequations}
\begin{align}
\nabla_z \mathcal L(\bar z, \bar\lambda) = 0,\\
\bar\lambda_i^g \geq 0, \; \bar\lambda_i^g g_i(\bar z) = 0,
&\;\forall i \in \{1,\dots,n_g\},\\ 
\bar\lambda_i^G\bar\lambda_i^H \geq 0,
&\; \forall i \in \beta(\bar z), \label{c-mult}\\
\left.
\begin{aligned}
\bar\lambda_i^H \geq 0, &\; \forall i \in \beta_1,\\
\bar\lambda_i^G \geq 0, &\; \forall i \in \beta_2,
\end{aligned}
\right\}
&\; \forall (\beta_1,\beta_2) \in \mathcal P(\beta(\bar z)).
\label{pw-a-mult}
\end{align} 
\end{subequations}
Noting that (\ref{c-mult}) and (\ref{pw-a-mult}) imply
(\ref{pw-m-mult}), we obtain that $\bar z$ is piecewise M-stationary
and therefore B-stationary (according to Theorem \ref{thm:suf}).
\end{proof}

\subsection{Convergence with Verifiable Conditions}

We have investigated stationarity of a limit point of the BA method
based on the condition (\ref{kktba}).
This condition holds if and only if (see \cite[Lemma 4.2]{gcq})
\begin{equation} \label{stat-geo}
\nabla f(z^k) \in
\mathcal T_{\rm BA}^{\rm lin}(z^k)^{*}.
\end{equation}
This section discusses the convergence results by introducing
conditions more readily verifiable than
(\ref{stat-geo}).

To this end, consider a sequence $\{z^k\} \to \bar z$ as $\epsilon^k \to 0$, where every
$z^k$ satisfies the Fritz John necessary condition for
BA($\epsilon^k$), namely, there exist multipliers 
$u^k = (u_0^k, u^{g,k}, u^{h,k}, u^{\Phi,k}) \neq 0$ with $u_0^k,u^{g,k} \geq 0$
such that 
\begin{equation} \label{fjba}
  0 = u_0^k\nabla f(z^k) 
      + \sum_{i \in I_g(z^k)} u_i^{g,k} \nabla g_i(z^k) 
      + \sum_{i \in I_h(z^k)} u_i^{h,k} \nabla h_i(z^k) 
      - \sum_{i =1}^{m} u^{\Phi,k}_i \nabla \Phi_i^{\epsilon}(z^k).
\end{equation}

\subsubsection{Results under MPCC-LICQ}
Assume MPCC-LICQ holds at $\bar z$ (i.e., NLP-LICQ holds for TNLP).
That is, the following gradients
are linearly independent:
\begin{equation*}
\begin{aligned}
  \nabla g_i(\bar z), &\quad i \in I_g(\bar z), \\
  \nabla h_i(\bar z), &\quad i \in I_h(\bar z), \\
  \nabla G_i(\bar z), &\quad i \in \alpha(\bar z) \cup \beta(\bar z), \\
  \nabla H_i(\bar z), &\quad i \in \gamma(\bar z) \cup \beta(\bar z).
\end{aligned}
\end{equation*}
For every $z^k$ in a neighborhood of $\bar z$, we have $I_g(z^k)
\subseteq I_g(\bar z), I_h(z^k)=I_h(\bar z), \{1,\dots,m\} =
\alpha(\bar z) \cup \gamma(\bar z) \cup \beta(\bar z)$, and therefore the gradients
\begin{equation*}
\begin{aligned}
  \nabla g_i(z^k), &\quad i \in I_g(z^k), \\
  \nabla h_i(z^k), &\quad i \in I_h(z^k), \\
  \nabla_G \Phi_i^{\epsilon}(z^k)\nabla G_i(\bar z) + 
  \nabla_H \Phi_i^{\epsilon}(z^k)\nabla H_i(z^k),
  &\quad i \in \{1,\dots,m\},
\end{aligned}
\end{equation*}
are also linearly independent, indicating that
NLP-LICQ holds for BA($\epsilon^k$) at $z^k$ for $k$ sufficiently large.
This guarantees $u_0^k \neq 0$, which recovers the condition
(\ref{kktba}).
C-stationarity of $\bar z$ follows directly from the
discussion in Section \ref{sec:cstat}, by
noting that the multipliers $w^k=u^k$ because of the presence of
NLP-LICQ.
Introducing additionally the second-order condition (\ref{rhess}),
piecewise M-stationarity of $\bar z$ follows from the discussion in Section
\ref{sec:pw-mstat}.
Moreover, $\bar z$ is S-stationary because MPCC-LICQ indicates
that the piecewise M-stationarity is satisfied
by the unique MPCC multipliers, which are the same for each branch of 
$\mathcal{P}(\beta(\bar{z})$, and are determined by:
\begin{equation} \label{wmult-u}
\begin{aligned} 
  \bar \lambda_i^g &= \bar u_i^g = \lim_{k \to \infty} u_i^{g,k},
  \; i \in I_g(\bar z),\\
  \bar \lambda_i^h &= \bar u_i^h = \lim_{k \to \infty} u_i^{h,k},
  \; i \in I_h(\bar z),\\
  \bar \lambda_i^G &= \left\{
    \begin{aligned}
      &\bar u_i^{\Phi} = \lim_{k \to \infty} u_i^{\Phi,k},
      &&i \in \alpha(\bar z), \\ 
      &\bar u_i^{\Phi} \theta_i, 
      &&i \in \beta(\bar z),
    \end{aligned}
  \right. \\
  \bar \lambda_i^H &= \left\{
    \begin{aligned}
      &\bar u^{\Phi}_i = \lim_{k \to \infty} u_i^{\Phi,k},
      &&i \in \gamma(\bar z), \\
      &\bar u^{\Phi}_i (1-\theta_i),
      &&i \in \beta(\bar z), \\
    \end{aligned}
  \right.
\end{aligned}
\end{equation}
for some $\theta_i \in [0,1]$. 
For clarity, we have the following corollary.

\begin{cor}
For a sequence of positive scalars $\epsilon^k \to 0$, apply the
Bounding Algorithm to BA($\epsilon^k$).
\begin{enumerate}[label = (\roman*)]
\item
Assume this generates a sequence
$\{z^k\} \to \bar z$, where every $z^k$ satisfies the Fritz John
necessary condition (\ref{fjba}) for
BA($\epsilon^k$), and MPCC-LICQ holds at $\bar z$.
Then $\bar z$ is C-stationary.
\item
Assume in addition that the condition (\ref{rhess}) holds at every
$z^k$ for $k$ sufficiently large. 
Then $\bar z$ is piecewise M-stationary and therefore S-stationary.
\end{enumerate}
\end{cor}

\subsubsection{Results under MPCC-MFCQ}

Assume that MPCC-MFCQ holds at $\bar z$ (i.e., NLP-MFCQ holds for TNLP).
That is, the following gradients are linearly independent:
\begin{equation*} 
\begin{aligned}
  \nabla h_i(\bar z), &\quad i \in I_h(\bar z), \\
  \nabla G_i(\bar z), &\quad i \in \alpha(\bar z) \cup \beta(\bar z),\\
  \nabla H_i(\bar z), &\quad i \in \gamma(\bar z) \cup \beta(\bar z),
\end{aligned}
\end{equation*}
and there exists a vector $d \in \mathbb R^n$ such that
\begin{equation*}
\begin{aligned}
  \nabla h_i(\bar z)^Td &= 0, &&\forall i \in I_h(\bar z), \\
  \nabla G_i(\bar z)^Td &=0, &&\forall i \in \alpha(\bar z) \cup \beta(\bar z), \\
  \nabla H_i(\bar z)^Td &=0, &&\forall i \in \gamma(\bar z) \cup \beta(\bar z), \\
  \nabla g_i(\bar z)^Td &<0, &&\forall i \in I_g(\bar z).
\end{aligned}
\end{equation*}
MPCC-MFCQ holds at $\bar z$ if and only if the active constraint gradients at $\bar z$ are positive-linearly
independent, namely, there do
not exist multipliers
$(\bar\lambda^g,\bar\lambda^h,\bar\lambda^G,\bar\lambda^H) \neq 0$
with $\bar\lambda^g \geq 0$ such that
\begin{equation*} 
\begin{aligned}
0 &= \sum_{i \in I_g(\bar z)} \bar\lambda^g_i \nabla g_i(\bar z)
 + \sum_{i \in I_h(\bar z)} \bar\lambda_i^h \nabla h_i(\bar z) \\
 &- \sum_{i \in \alpha(\bar z) \cup \beta(\bar z)} \bar\lambda^G_i
 \nabla G_i(\bar z) 
 - \sum_{i \in \gamma(\bar z) \cup \beta(\bar z)} \bar\lambda^H_i
 \nabla H_i(\bar z).
\end{aligned}
\end{equation*}
According to \cite[Proposition 2.2]{cpld}, for every $z^k$ in a
neighborhood of $\bar z$, there do not exist 
$(v^{g,k},v^{h,k},v^{G,k},v^{H,k}) \neq 0$
with $v^{g,k} \geq 0$ such that
\begin{equation} \label{pld-zk}
\begin{aligned}
0 &= \sum_{i \in I_g(\bar z)} v^{g,k}_i \nabla g_i(z^k)
 + \sum_{i \in I_h(\bar z)} v_i^{h,k} \nabla h_i(z^k) \\
 &- \sum_{i \in \alpha(\bar z) \cup \beta(\bar z)} v^{G,k}_i \nabla G_i(z^k) 
  - \sum_{i \in \gamma(\bar z) \cup \beta(\bar z)} v^{H,k}_i \nabla H_i(z^k).
\end{aligned}
\end{equation}
If NLP-MFCQ fails for BA($\epsilon^k$) at $z^k$, then there
exist multipliers $(u^{g,k},u^{h,k},u^{\Phi,k}) \neq 0$
with $u^{g,k} \geq 0$ such that
\begin{equation*} 
\begin{aligned}
  0 &= \sum_{i \in I_g(z^k)} u_i^{g,k} \nabla g_i(z^k) 
      + \sum_{i \in I_h(z^k)} u_i^{h,k} \nabla h_i(z^k) 
      - \sum_{i =1}^{m} u^{\Phi,k}_i \nabla \Phi_i^{\epsilon}(z^k) \\
    &= \sum_{i \in I_g(z^k)} u_i^{g,k} \nabla g_i(z^k) 
      + \sum_{i \in I_h(z^k)} u_i^{h,k} \nabla h_i(z^k) \\
     &- \sum_{i \in \alpha(\bar z)} u^{\Phi,k}_i \nabla G_i(z^k)
      - \sum_{i \in \beta(\bar z)}
        u^{\Phi,k}_i \nabla_G \Phi_i^{\epsilon}(z^k)\nabla G_i(z^k)\\
     &- \sum_{i \in \gamma(\bar z)} u^{\Phi,k}_i \nabla H_i(z^k)
      - \sum_{i \in \beta(\bar z)}
        u^{\phi,k}_i \nabla_H\Phi_i^{\epsilon}(z^k)\nabla H_i(z^k).
\end{aligned}
\end{equation*}
Thus (\ref{pld-zk}) holds with
\begin{equation*}
\begin{aligned}
  v_i^{g,k} &= \left\{
    \begin{aligned}
      &u_i^{g,k}, &&i \in I_g(z^k), \\
      &0, &&i \in I_g(\bar z) \setminus I_g(z^k),
    \end{aligned}
  \right. \\
  v_i^{h,k} &= u_i^{h,k}, \; i \in I_h(\bar z), \\
  v_i^{G,k} &= \left\{
    \begin{aligned}
      &u_i^{\Phi,k}, &&i \in \alpha(\bar z), \\
      &u_i^{\Phi,k}\nabla_G\Phi_i^{\epsilon}(z^k), &&i \in \beta(\bar z),
    \end{aligned}
  \right. \\
  v_i^{H,k} &= \left\{
    \begin{aligned}
      &u_i^{\Phi,k}, &&i \in \gamma(\bar z), \\
      &u_i^{\Phi,k}\nabla_H\Phi_i^{\epsilon}(z^k), &&i \in \beta(\bar z),
    \end{aligned}
  \right.
\end{aligned}
\end{equation*}
which contradicts MPCC-MFCQ at $\bar z$.
Hence, NLP-MFCQ holds for BA($\epsilon^k$) at $z^k$ for $k$
sufficiently large.

Therefore, the condition (\ref{fjba}) cannot hold with $u_0^k=0$,
which recovers the condition (\ref{kktba}) and C-stationarity of $\bar
z$ follows.
The MPCC multipliers (\ref{wmult-u}) are still valid, because
the known boundedness of the multipliers $u^k$ allows to establish
the convergence $\{u^k\} \to \bar \lambda$ directly from (\ref{kktba}).

Assume furthermore for the purpose of obtaining a contradiction that $\bar z$ is not piecewise M-stationary, even if 
the second-order condition
\begin{equation} \label{rhess-u}
  q^T \nabla_{zz} \mathcal L_{\rm BA}(z^k,u^k) q > -\infty, \quad
  \forall q \in \mathcal Q(z^k)
\end{equation}
with
\begin{equation*}
\begin{aligned}
    \mathcal Q(z^k) = \{ q \,|\,
      &\nabla g_i(z^k)^Tq = 0, &&\forall i \in I_g(z^k)
      \text{ with } u_i^{g,k} >0,\\
      &\nabla g_i(z^k)^Tq \leq 0, &&\forall i \in I_g(z^k)
      \text{ with } u_i^{g,k} =0,\\
    ~ &\nabla h_i(z^k)^Tq = 0, &&\forall i \in I_h(z^k),\\
    ~ &\nabla \Phi_i^{\epsilon}(z^k)^Tq = 0,
      &&\forall i \in \{1,\dots,m\}\}
\end{aligned}
\end{equation*}
holds at every $z^k$.
Note that the condition (\ref{rhess-u}) is in terms of multipliers
$u^k$ and directions $q \in \mathcal Q(z^k)$ and usually can be detected
by NLP solvers employing second-order information.

As before, the assumption that $\bar z$ is not piecewise
M-stationary implies (\ref{caseb1}) or (\ref{caseb2}).
Starting from (\ref{caseb1}) we arrive at (\ref{rhessphi}), which implies
\begin{equation} \label{rhessphi-u}
- u_{i_1}^{\Phi,k} (d^{\hat\beta_1})^T\nabla_{zz}\Phi^{\epsilon}_{i_1}(z^k) d^{\hat\beta_1}
\to -\infty, \;\exists d^{\hat\beta_1} \in \mathcal D(z^k),
\end{equation}
because $w^k$ and $u^k$ have the same sign (as shown by (\ref{wk})).
We relate $d^{\hat\beta_1}$ to a bounded direction $q^{\hat\beta_1} \in \mathcal Q(z^k)$. 
Denote the system (\ref{dir1}) as $A^Td^{\hat\beta_1} = b$.
Let
\begin{equation*}
q^{\hat\beta_1} = d^{\hat\beta_1} + d_{\rm null}^{\hat\beta_1}.
\end{equation*}
Since the matrix $A$ is of full column rank, we can choose
\begin{equation*}
\begin{aligned}
  d^{\hat\beta_1} &= A(A^TA)^{-1}b,
  &&(\in \text{range}(A)) \\
  d_{\rm null}^{\hat\beta_1} &= (I - A(A^TA)^{-1}A^T)q^{\hat\beta_1}.
  &&(\in \text{null}(A^T))
\end{aligned}
\end{equation*}
It follows that
\begin{equation*}
  \nabla G_{i1}(z^k)^T d_{\rm null}^{\hat\beta_1} = 0, \;
  \nabla H_{i1}(z^k)^T d_{\rm null}^{\hat\beta_1} = 0,
\end{equation*}
and 
\begin{equation*}
\begin{aligned}
- u_{i_1}^{\Phi,k} (q^{\hat\beta_1})^T\nabla_{zz}\Phi^{\epsilon}_{i_1}(z^k) q^{\hat\beta_1}
= &- u_{i_1}^{\Phi,k} (d^{\hat\beta_1})^T\nabla_{zz}\Phi^{\epsilon}_{i_1}(z^k) d^{\hat\beta_1}
- u_{i_1}^{\Phi,k} (d_{\rm
  null}^{\hat\beta_1})^T\nabla_{zz}\Phi^{\epsilon}_{i_1}(z^k) d_{\rm null}^{\hat\beta_1} \\
= &- u_{i_1}^{\Phi,k} (d^{\hat\beta_1})^T\nabla_{zz}\Phi^{\epsilon}_{i_1}(z^k) d^{\hat\beta_1}\\
  &- u_{i_1}^{\Phi,k} (d_{\rm null}^{\hat\beta_1})^T\nabla_G\Phi^{\epsilon}_{i_1}(z^{\hat\beta_1})
   \nabla_{zz}G_{i_1}(z^k)d_{\rm null}^{\hat\beta_1} \\
  & -u_{i_1}^{\Phi,k} (d_{\rm null}^{\hat\beta_1})^T\nabla_H\Phi^{\epsilon}_{i_1}(z^k)
   \nabla_{zz}H_{i_1}(z^k)d_{\rm null}^{\hat\beta_1} .
\end{aligned}
\end{equation*}
In the last equality, since the last two terms are bounded, we obtain
from (\ref{rhessphi-u}) that
\begin{equation} \label{rhessphi-u2}
- u_{i_1}^{\Phi,k} (q^{\hat\beta_1})^T\nabla_{zz}\Phi^{\epsilon}_{i_1}(z^k) q^{\hat\beta_1}
\to -\infty, \;\exists q^{\hat\beta_1} \in \mathcal Q(z^k).
\end{equation}
Since all other terms in the reduced Hessian $(q^{\hat\beta_1})^T
\nabla_{zz}\mathcal L_{\rm BA}(z^k,u^k)q^{\hat\beta_ 1}$ are bounded,
we obtain that (\ref{rhessphi-u2}) contradicts (\ref{rhess-u}). Hence
(\ref{caseb1}) must be false. Similarly, we can show that
(\ref{caseb2}) must also be false, by choosing a direction
$d^{\hat\beta_2} \in \mathcal D(z^k)$ and relating it to a bounded direction $q^{\hat\beta_2} \in
\mathcal Q(z^k)$. Therefore, our initial assumption does not hold and
$\bar z$ is piecewise M-stationary. 

As a result, we have the following corollary.

\begin{cor}
For a sequence of positive scalars $\epsilon^k \to 0$, apply the
Bounding Algorithm to BA($\epsilon^k$).
\begin{enumerate}[label = (\roman*)]
\item
Assume this generates a sequence
$\{z^k\} \to \bar z$, where every $z^k$ satisfies the Fritz John
necessary condition (\ref{fjba}) for
BA($\epsilon^k$), and MPCC-MFCQ holds at $\bar z$.
Then $\bar z$ is C-stationary.
\item
Assume in addition that the condition (\ref{rhess-u}) holds at every
$z^k$ for $k$ sufficiently large. 
Then $\bar z$ is piecewise M-stationary and therefore B-stationary.
\end{enumerate}
\end{cor}

\subsubsection{Results under observation of bounded multipliers}

In numerical experiments, it is usually convenient to observe whether
the multipliers $u^k$ associated with $z^k$ are bounded or not.
Observation of a sequence of bounded multipliers $u^k$ as $\epsilon^k \to 0$, is a
practical certificate that every $z^k$ is not only a Fritz John point
but a KKT point of
BA($\epsilon^k$), namely, the condition (\ref{kktba}) holds.
Although such observation only shows the existence of bounded
multipliers at $z^k$, which is weaker than what we could expect under
NLP-MFCQ, the validity of (\ref{kktba}) still allows to apply the
Caratheodory's theorem for cones. As a result, a sequence of
subset $J^k = (J_g^k, J_h^k, J_{\Phi}^k)$ of linearly independent active gradients at $z^k$, and nonzero
multipliers $w^k$ uniquely determined by $J^k$, can be constructed as
in Sections \ref{sec:cstat} and \ref{sec:pw-mstat}.
This leads to C-stationarity of a limit point $\bar z$.
The multipliers (\ref{wmult-u}) are valid because $u^k$ are bounded
and the convergence $\{u^k\} \to \bar \lambda$ can be derived directly
from (\ref{kktba}).
Moreover, with the additional condition (\ref{rhess-u}) and
boundedness of $u^k$, $\bar z$ is piecewise M-stationary (the
derivation is the same as that under MPCC-MFCQ).

The following corollary summarizes the discussion.

\begin{cor}
For a sequence of positive scalars $\epsilon^k \to 0$, apply the
Bounding Algorithm to BA($\epsilon^k$).
\begin{enumerate}[label = (\roman*)]
\item
Assume this generates a sequence
$\{z^k\} \to \bar z$, where every
$z^k$ satisfies the Fritz John necessary condition
(\ref{fjba}) for BA($\epsilon^k$), and it is observed that the multipliers $u^k$ at every $z^k$ are bounded.
Then $\bar z$ is C-stationary.
\item
Assume in addition that the condition (\ref{rhess-u}) holds at every
$z^k$ for $k$ sufficiently large. 
Then $\bar z$ is piecewise M-stationary and therefore B-stationary.
\end{enumerate}
\end{cor}

\end{document}